\documentclass{amsart}
\usepackage{graphicx}
\usepackage{amssymb,latexsym, hyperref}
\usepackage{a4wide}
\usepackage{psfrag}
\usepackage[all]{xy}
\xyoption{2cell}
\UseTwocells

\let\cal\mathcal
\def\AA{{\cal A}}
\def\BB{{\cal B}}
\def\CC{{\cal C}}
\def\DD{{\cal D}}

\def\FF{{\cal F}}
\def\GG{{\cal G}}

\def\KK{{\cal K}}

\def\NN{{\cal N}}
\def\OO{{\cal O}}
\def\PP{{\cal P}}

\def\SS{{\cal S}}
\def\TT{{\cal T}}
\def\UU{{\cal U}}
\def\VV{{\cal V}}
\def\WW{{\cal W}}

\let\blb\mathbb

\def\bX{{\blb X}}

\def\bP{{\blb P}}

\def\bZ{{\blb Z}}

\def\bN{{\blb N}}

\def\bZ{{\blb Z}}

\def\bK{{\blb K}}

\let\frak\mathfrak

\def\pp{\frak{p}}

\newcommand{\proj}{\operatorname{proj}}

\def\Mod{\operatorname{Mod}}

\def\mod{\operatorname{mod}}

\def\Supp{\mathop{\text{\upshape Supp}}}

\def\coh{\mathop{\text{\upshape{Coh}}}}

\def\Spec{\operatorname {Spec}}

\def\rep{\operatorname{rep}}

\def\Ext{\operatorname {Ext}}

\def\Hom{\operatorname {Hom}}

\def\End{\operatorname {End}}
\def\RHom{\operatorname {RHom}}

\def\im{\operatorname {im}}

\def\cone{\operatorname {cone}}
\def\coker{\operatorname {coker}}
\def\ker{\operatorname {ker}}

\def\End{\operatorname {End}}

\def\rk{\operatorname {rk}}

\DeclareMathOperator{\thick}{thick}
\DeclareMathOperator{\wide}{wide}

\newcommand\Db{D^{b}}
\newcommand\Kb{K^{b}}

\newcommand{\PW}{\PP\WW}

\renewcommand\t{\tau}

\newtheorem{lemma}{Lemma}[section]
\newtheorem{proposition}[lemma]{Proposition}
\newtheorem{theorem}[lemma]{Theorem}
\newtheorem{corollary}[lemma]{Corollary}

\theoremstyle{definition}

\newtheorem{example}[lemma]{Example}
\newtheorem{definition}[lemma]{Definition}

{

}

\theoremstyle{remark}

\newtheorem{remark}[lemma]{Remark}

\newdimen\uboxsep \uboxsep=1ex
\def\uboxn#1{\vtop to 0pt{\hrule height 0pt depth 0pt\vskip\uboxsep
\hbox to 0pt{\hss #1\hss}\vss}}

\def\uboxs#1{\vbox to 0pt{\vss\hbox to 0pt{\hss #1\hss}
\vskip\uboxsep\hrule height 0pt depth 0pt}}

\def\Ob{\operatorname{Ob}}

\def\Gen{\operatorname{Gen}}

\newcommand\exa{\nopagebreak \begin{center}\smallskip \nopagebreak               \begin{minipage}[t]{6cm}\sloppy}
\newcommand\exb{\end{minipage}\kern 1cm\begin{minipage}[t]{8cm}\sloppy}
\newcommand\exc{\end{minipage}\kern -3cm \smallskip\end{center}}

\title{$t$-structures for hereditary categories}
\author{Donald Stanley}
\address{Donald Stanley, Dept. of Math. \& Stats. \\ University of Regina \\ Regina, Canada S4S 4A5}\email{donald.stanley@uregina.ca}
\author{Adam-Christiaan van Roosmalen}
\address{Adam-Christiaan van Roosmalen, Dept. of Math. \& Stats. \\ University of Regina \\ Regina, Canada S4S 4A5}\email{vanroosa@uregina.ca}

\begin{document}
\date{\today}

\bibliographystyle{amsplain}

\begin{abstract}
We study aisles in the derived category of a hereditary abelian category.  Given an aisle, we associate a sequence of subcategories of the abelian category by considering the different homologies of the aisle.  We then obtain a sequence, called a narrow sequence.  

We then prove that a narrow sequence in a hereditary abelian category consists of a nondecreasing sequence of wide subcategories, together with a tilting torsion class in each of these wide subcategories.  Furthermore, there are relations these torsion classes have to satisfy.  These results are sufficient to recover known classifications of $t$-structures for smooth projective curves, and for finitely generated modules over a Dedekind ring.

In some special cases, including the case of finite dimensional modules over a finite dimensional hereditary algebra, we can reduce even further, effectively decoupling the different tilting torsion theories one chooses in the wide subcategories.
\end{abstract}

\maketitle

\tableofcontents

\section{Introduction}

Let $R$ be a commutative noetherian ring.  In \cite{Stanley10}, an invariant of an aisle $\UU \subseteq \Db \mod R$ was given as a function
$$\varphi_\UU: \bZ \longrightarrow \{\mbox{Subsets of $\Spec R$ closed under specialization}\}.$$
For any $n \in \bZ$ one obtains $\varphi_\UU(n)$ as follows.  First truncate $\UU$ above $n$ with the standard truncation to obtain $\t^{\geq -n} \UU$ and then define $\varphi_\UU(n)$ to be the thick closure of $\tau^{\geq -n} \UU$.  By results of Hopkins and Neeman, these thick subcategories are classified exactly by the subsets of $\Spec R$ which are closed under specialization.  If the ring $R$ is nice enough (for example, if $\Db \mod R$ admits a dualizing complex), then one can use the above function to classify all $t$-structures in $\mod R$ (see \cite{TarrioLeovigildoJeremiasSaorin10}).

In this paper, we want to change the setting of affine algebraic geometry to a more noncommutative geometric setting, replacing $\mod R$ by an abelian category $\AA$.  This category $\AA$ can be thought of as the category of (quasi-)coherent sheaves on some noncommutative variety.  For the rest of this introduction, we will restrict ourselves to the case where $\AA$ is hereditary (i.e. has global dimension at most one).  This then includes the cases where $\AA$ is the category of (quasi-)coherent sheaves on a smooth projective curve, or when $\AA$ is the category of (finite dimensional) modules over a (finite dimensional) hereditary algebra.

In general however, it is not clear whether there is a geometric object, such as $\Spec R$, to classify against.  Moreover, knowing the thick closures of $\tau^{\geq -n} \UU$ is in most cases not enough to recover $\UU$ (see for example \cite{GorodentsevKuleshovRudakov04} or Section \ref{section:ApplicationCurves}); the reason for this is that torsion classes and thick subcategories will not coincide in our setting as they do in the case of a noetherian commutative ring (see \cite{StanleyWang11}).

Thus let $\UU$ be a preaisle in $\Db \AA$, and instead of considering the thick closures of $\tau^{\geq -n} \UU$, we will consider the $n^\text{th}$ homology $H_n \UU$ as a subcategory of $\AA$.  In order for this transition between subcategories of the derived category and subcategories of the abelian category to work smoothly, we will restrict ourselves to homology-determined preaisles.  A preaisle $\UU$ is said to be \emph{homology-determined} if $X \in \UU$ is equivalent to $H_n(X) \in H_n(\UU)$, for all $n \in \bZ$.  For a hereditary category $\AA$, an aisle in $\Db \AA$ is always homology-determined (Proposition \ref{proposition:RetractsCoproducts}) but preaisles are not necessarily homology-determined (Example \ref{example:BadPreaisles}).  When $\AA$ is not hereditary, then aisles need not be homology-determined (Example \ref{example:Preaisle}).

For a homology-determined preaisle in $\Db \AA$, one can describe each subcategory $H_n(\UU) \subseteq \AA$ as a \emph{narrow} subcategory of $\AA$ (\cite{StanleyWang11}), these are by definition full subcategories closed under cokernels and extensions.  When $\AA$ is hereditary, another characterization of a narrow subcategory is given in Corollary \ref{corollary:NarrowIsPretorsion} where it is shown that narrow subcategories are exactly given by tilting nullity classes in wide subcategories of $\AA$ (see Proposition \ref{proposition:GeneratedInOneStep}).  Recall that a subcategory of $\AA$ is called \emph{wide} if it is closed under kernels, cokernels, and extension.

A narrow sequence is a nondecreasing sequence of narrow subcategories, satisfying an additional condition (see Definition \ref{definition:NarrowSequence}).  Thus the concept of a narrow sequence is one in abelian categories.  The following theorem (Theorem \ref{theorem:NarrowPreaisle} in the text) relates this concept to homology-determined aisles, which are defined in the setting of derived categories.

\begin{theorem}\label{theorem:Introduction}
Let $\AA$ be an abelian category.  There are bijections
$$\begin{array}{lcr}
\left\{
\mbox{Narrow sequences in $\AA$}
\right\}&
\stackrel{\sim}{\longleftrightarrow}&
\left\{ \mbox{Homology-determined preaisles in $\Db \AA$} \right\}.
\end{array}$$
\end{theorem}

With some basic knowledge of the category $\coh \bP^1$ of coherent sheaves on a projective line, one can use this theorem to classify all $t$-structures on $\coh \bP^1$ (see Corollary \ref{corollary:AislesProjectiveLine}).  We thus recover the classification given in \cite{GorodentsevKuleshovRudakov04}.

One can now wonder which narrow sequences $(\NN(n))_{n \in \bZ}$ correspond to aisles under the bijection in Theorem \ref{theorem:Introduction}.  One quickly observes that each narrow subcategory must be coreflective, i.e. the embedding $\NN(n) \to \AA$ has a right adjoint.  This, however, is not sufficient.  One has to be able to ``glue'' these right adjoints together to form a right adjoint of the embedding of the corresponding preaisle $\UU$ into $\Db \AA$.  We will show that this can be done when, for example, $\AA$ has enough injectives.  The following theorem is Theorem \ref{theorem:EnoughInjectives}.

\begin{theorem}\label{theorem:IntroductionInjectives}
Let $\AA$ be a hereditary category with enough injectives.  There are bijections
$$\begin{array}{lcr}
\left\{ \mbox{$t$-structures on $\Db \AA$} \right\}&
\stackrel{\sim}{\longleftrightarrow}&
\left\{ \mbox{Coreflective narrow sequences in $\AA$} \right\}.
\end{array}$$
\end{theorem}

Theorem \ref{theorem:IntroductionInjectives}, together with Proposition \ref{proposition:SequencesAisles} which leads up to this theorem, is sufficient to describe the $t$-structures on $\coh \bX$ where $\bX$ is a smooth projective curve of genus at least one (see Theorem \ref{theorem:AislesCurves}), and on $\mod R$ where $R$ is a Dedekind domain (see Theorem \ref{theorem:Dedekind}).  The first case has already been covered in \cite{GorodentsevKuleshovRudakov04} and the second case follows from \cite{TarrioLeovigildoJeremiasSaorin10}.

There is some redundant information in the definition of a narrow sequence $(\NN(k))_{k \in \bZ}$.  For example, it follows from Corollary \ref{corollary:GrowingFastEnough} that $\NN(n+1)$ must contain the wide closure of $\NN(n)$.  Assume now that the abelian category satisfies the following property: for each aisle $\UU$, the thick subcategory of $\Db \AA$ generated by $\tau^{\geq -n} \UU$ is both reflective and coreflective, i.e. the embedding of the thick closure of $\tau^{\geq -n} \UU$ into $\Db \AA$ has both a left and a right adjoint.  The right adjoint is automatic if $\AA$ has enough injectives; the left adjoint is present, for example, when $\AA \cong \rep Q$ where $\rep Q$ is the category of finite dimensional representations of a finite acyclic quiver $Q$.

When $\AA$ satisfies the additional condition above, one can further simplify the data required to classify aisles in $\Db \AA$.  Instead of a narrow sequence, where the narrow subcategories are linked together by exact sequences as in Definition \ref{definition:NarrowSequence}, it suffices to give the wide closure $\WW(k)$ of each $\NN(k)$, together with a tilting torsion class in $\WW(k) \cap {}^\perp \WW(k-1)$, this tilting torsion class is $\NN(k) \cap {}^\perp \WW(k-1)$.  There are two main advantages of this approach.  The first is that wide subcategories and tilting torsion theories are more studied, and hence better understood, than narrow subcategories.  The second advantage is that any nondecreasing sequence of wide subcategories will suffice for $(\WW(k))_{n \in \bZ}$, and that there are no restrictions on the chosen tilting torsion class induced by the choice of the other tilting torsion classes.  Thus where the sequence $(\NN(k))_{k \in \bZ}$ satisfies some extra conditions, the new sequence is essentially decoupled.

For the category $\AA$, we define $\Delta(\AA)$ to consist of the pairs $(f,t_f)$ where $(f(n) \subseteq \AA)_{n \in \bZ}$ is a nondecreasing sequence of coreflective wide subcategory, and where $t_f(n) \subseteq f(n) \cap {}^\perp f(n-1)$ is a tilting torsion theory.  Our main theorem is then the following (see Theorem \ref{theorem:Reduced}).

\begin{theorem}
Let $\AA$ be a hereditary category such that for every coreflective wide subcategory $\WW$ of $\AA$, the canonical functor $\Db \WW \to \Db \AA$ has a left and a right adjoint.  Then the functions $\Xi$ and $\Psi$ from Definitions \ref{definition:Xi} and \ref{definition:Psi} respectively yield bijections
$$\begin{array}{lcr}
\{\mbox{$t$-structures on $\Db \AA$}\}&
\stackrel{\sim}{\longleftrightarrow}&
\Delta(\AA).
\end{array}$$
\end{theorem}

As stated before, the main application of this theorem is the case where $\AA$ is the category of finite dimensional representations of a finite dimensional hereditary algebra over a field.  The in-depth discussion will be given in a follow-up paper.

\textbf{Acknowledgments} The authors would like to thank Jan {\v{S}}{\v{t}}ov{\'{\i}}{\v{c}}ek for meaningful discussion and especially for pointing out Proposition \ref{proposition:Neeman} to us.
The second author also gratefully acknowledges the support of the Hausdorff Center for Mathematics in Bonn and Bielefeld University.
%\section{Notations and conventions}
%
%We will choose a Grothendieck universe $\mathfrak{U}$, and we assume that all our categories are $\mathfrak{U}$-categories, i.e. the Hom-sets are elements of $\mathfrak{U}$.
%
%Let $\AA$ be an abelian category.  We denote by $D^* \AA$ its bounded ($*=b$), unbounded ($*=\emptyset$), left bounded ($*=+$), or right bounded ($*=-$) derived category.  The category $D^* \AA$ is a triangulated category, and we will denote the $n$-fold suspension by $[n]$.  We will often call an exact triangle in $D^* \AA$ just a triangle.
%
%There is a fully faithful functor $\AA \to D^* \AA$ which maps an object of $\AA$ to a stalk complex in $D^* \AA$ concentrated in degree $0$.  When $A \in \AA$, the corresponding stalk complex will be denoted by $A[0]$.
%
%We will write $H^n: D^* \AA \to \AA$ for the usual $n^{\text {th}}$-cohomology functor.  Furthermore we will write $H_n = H^{-n}$.  For any object $X \in D^* \AA$, we will write $X_n = H_n X$, and $X^n = H^n X$.  Similarly, if $\CC \subseteq D^* \AA$ is any full subcategory, then $\CC_n$ and $\CC^n$ are the essential images of $H_n \CC$ and $H^n \CC$, respectively.
%
%%Let $\CC \subseteq D^* \AA$ be any full subcategory.  We will write $\wide_n \CC$ for the full wide subcategory of $\AA$ generated by $\CC_n$.  We will write $\thick_n \CC$ for the thick subcategory of $D^* \AA$ generated by $H_n \CC[n] \subseteq D^* \AA$.  Thus $\thick_n \CC$ is canonically equivalent to $\Db (\wide_n \CC)$ as a full subcategory of $D^* \AA$.
%
%An abelian category will be called hereditary if $\Ext_\AA^2(-,-) = 0$.

\section{Background and notation}

In this section, we introduce some notation and recall some known properties.  Some result will be used in a slightly different setting than usually, and we provide proofs for those results.  In particular, we would like to point out that Section \ref{subsection:tStructures} contains mainly results from \cite{KellerVossieck88}, but we will use Proposition \ref{proposition:Neeman} (originally from \cite{Neeman10}) to obtain these results in a slightly more general setting.

Almost all the result in this section are likely known 

The results in this section are mostly well-known, or easily proved using 

\subsection{Adjoint functors}

Let $F: \CC \to \DD$ be (covariant) functor.  We will say the functor $R: \DD \to \CC$ is \emph{right adjoint} to $F$ if there are bijections
$$\Hom_\DD(FC,D) \stackrel{\eta_{C,D}}{\longrightarrow} \Hom(C,RD)$$
natural in both components.  If $R$ is right adjoint to $F$, then we will say that $F$ is \emph{left adjoint} to $R$.  We have the following equivalent formulations.

\begin{proposition}
Let $F: \CC \to \DD$ be a functor.  The following are equivalent.
\begin{enumerate}
\item The functor $F$ has a right adjoint.
\item The functor $\Hom_\DD(F-,D)$ is representable for each $D \in \DD$.
\item For any object $D \in \DD$ there is an object $D' \in \CC$ and a map $\varphi_D: FD' \to D$ such that for any $C \in \CC$ there is a bijection $\varphi_D \circ F(-): \Hom_\CC(C,D') \stackrel{\sim}{\rightarrow} \Hom_\DD(FC,D)$.
\end{enumerate}
\end{proposition}

\begin{proof}
We refer to \cite{Maclane71} for the proof.
\end{proof}

The map $\varphi_D: FD' \to D$ from the previous proposition satisfies a universal property.  We will refer to it as the \emph{universal map}.

Let $\CC$ be a full replete (= closed under isomorphisms) subcategory of $\DD$.  We will say that $\DD$ is \emph{reflective} or \emph{coreflective} if and only if the embedding $\CC \to \DD$ has a left or a right adjoint, respectively.  When $\CC$ is a coreflective subcategory of $\DD$, the right adjoint to $\CC \to \DD$ will be denoted by $(-)_\CC$, thus mapping a function $f: X \to Y$ to $f_\CC: X_\CC \to Y_\CC$.  Similarly, for a left adjoint we will write $(-)^\CC$.

In general, the intersection of two coreflective subcategories is not a coreflective subcategory.  We have the following proposition.

\begin{proposition}\label{proposition:IntersectionAdjoint}
Let $\CC,\CC'$ be coreflective subcategories of $\DD$.  If for every $D \in \DD$, we have $(D_\CC)_{\CC'} \in \CC$ then $\CC \cap \CC'$ is a coreflective subcategory of $\DD$.
\end{proposition}

\begin{proof}
Note that $\Hom(-,D)\mid_{\CC \cap \CC'}$ is represented by the object $(D_\CC)_{\CC'}$.
\end{proof}

\begin{lemma}\label{lemma:AdjointRetract}
Let $\CC$ be a full coreflective subcategory of $\DD$.  Then $\CC$ is closed under retracts.
\end{lemma}

\begin{proof}
Denote the embedding $\CC \to \DD$ by $F$ and let $(-)_\CC$ be its right adjoint.  Let $C \in \CC$ and let $D \in \DD$ be a retract of $FC$.  We want to show that $D$ lies in the essential image of $F$.  By \cite[Proposition 6.5.4]{Borceux94}, we know that $D \to FC$ is the equalizer of
$$\xymatrix@1{FC \ar@<2pt>[r]^{e} \ar@<-2pt>[r]_{1}& FC}$$
where $e \in \End(FC)$ is an idempotent.  Since $(FC)_\CC \cong C$ ($F$ is fully faithful) and $(-)_\CC$ commutes with limits ($(-)_\CC$ has a left adjoint), we know that $D_\CC$ is an equalizer of
$$\xymatrix@1{C \ar@<2pt>[r]^{(e)_\CC} \ar@<-2pt>[r]_{1}& C}$$
where now $(e)_\CC \in \End(C)$ is an idempotent.  Hence $D_\CC \in \CC$ is a retract of $C$, and thus also the coequalizer of the second diagram.  Since $F$ commutes with colimits, we know that $F(D_\CC)$ is the equalizer of the first diagram and hence $F(D_\CC) \cong D$ so that $D$ lies in the essential image of $F$ as required.
\end{proof}

\subsection{Split injectives and Ext injectives}\label{section:Injectives}

Let $\AA$ be an abelian category and let $\CC$ be any full subcategory.  An object $I \in \CC$ is called \emph{$\CC$-split injective} if and only if every monomorphism $I \to M$ in $\AA$ splits ($M \in \CC$).  We will say that $I \in \CC$ is \emph{$\CC$-Ext injective} if and only if $\Ext_\AA(X,I) = 0$, for all $X \in \CC$.

\begin{proposition}
Let $I \in \Ob \CC$ be $\CC$-Ext injective and $I' \in \Ob \CC$ is a retract of $I$, then $I'$ is $\CC$-Ext injective.
\end{proposition}

\begin{proof}
Since $\Ext_\AA(-,I')$ is a subfunctor of $\Ext_\AA(-,I)$ and $\Ext_\AA(-,I)\mid_\CC \cong 0$, we have that $\Ext_\AA(-,I')\mid_\CC \cong 0$.
\end{proof}

We have the following proposition.

\begin{proposition}\label{proposition:SplitExt}
Let $\CC$ be a full subcategory of $\AA$ which is closed under extensions and cokernels.  Then an object $I \in \Ob \CC$ is $\CC$-split injective if and only if it is $\CC$-Ext injective.
\end{proposition}

\begin{proof}
Let $I \in \Ob \CC$ be a $\CC$-split injective, we will show that $I$ is also a $\CC$-Ext injective object.  Let $X \in \CC$ be any object and consider a short exact sequence $0 \to I \to M \to X \to 0$ in $\AA$.  Since $I,X \in \Ob \CC$ and $\CC$ is closed under extensions, we know that $M \in \Ob \CC$.  Thus $I \to M$ splits, and we see that $\Ext_\AA(X,I) = 0$.

For the other direction, let $I \to M$ be a monomorphism.  Since $\CC$ is closed under cokernels, the exact sequence $0 \to I \to M \to X \to 0$ in $\AA$ has entries in $\CC$.  This is split because $I$ is $\CC$-Ext injective.  We find that the original monomorphism $I \to M$ splits.
\end{proof}

\begin{definition}
In the setting of Proposition \ref{proposition:SplitExt}, the object $I$ will be called a $\CC$-injective object.
\end{definition}

There is a dual notion of $\CC$-Ext projectives and $\CC$-split projectives.  If $\CC$ is closed under subobjects and extensions, then these two notions coincide.

\subsection{Derived categories}

We will follow standard notations and conventions about derived categories (see for example \cite{GelfandManin03, Hartshorne66,Keller07}).

Let $\AA$ be an abelian category and denote by $D \AA$, $D^+ \AA$, $D^- \AA$, and $D^b \AA$ the total, the left bounded, the right bounded, and the bounded derived category.  We will denote the $n$-fold suspension (or translation) by $[n]$, for all $n \in \bZ$.

There is a fully faithful function $\AA \to D^*(\AA)$ mapping $A$ to a complex which is the stalk complex of $A$ concentrated in degree 0 is; we will denote this complex by $A[0]$.  Also, we will write $(A[0])[n]$ by $A[n]$.

Taking homologies of objects in $\AA$ defines a homological functor $H: D^* \AA \to \AA$, mapping a triangle $X \to Y \to Z \to X[1]$ to a long exact sequence
$$\cdots \longrightarrow H_{k+1} Z \longrightarrow H_k X \longrightarrow H_k Y \longrightarrow H_k Z \longrightarrow H_{k-1} X \longrightarrow \cdots$$
Note that $H_k(X) = H_0(X[-k])$ and thus $H_k(X) = H^{-k}(X)$.  To avoid cumbersome notation, we will often write $H_k X = X_k$ for an object $X \in D^* \AA$.  Likewise, for a full subcategory $\CC \subseteq D^* \AA$, we will write $\CC_k$ for $H_k(\CC)$.

For a hereditary category $\AA$ it is well-known (see for example \cite{ChenKrause09}) that every object $A \in D^* \AA$ can be written as
$$\prod_{k \in \bZ} A_k[k] \cong A \cong \coprod_{k \in \bZ} A_k[k].$$
In particular, $\Hom_{D^*(\AA)}(A,B) \cong \prod_{k,l} \Ext^{l-k}_{\AA}(A_k,B_l).$

We will often use the following lemma.

\begin{lemma}\label{lemma:HereditaryAdjoints}
Let $\AA$ be a hereditary abelian category, and let $A \in \AA$.  Let $\CC$ be a coreflective subcategory of $D^* \AA$.  Then $H_l((A[k])_\CC) = 0$ for $l \not= k, k-1$.
\end{lemma}

\begin{proof}
To ease notation, we write $C = (A[k])_\CC$.  Recall that $\Hom_\CC(-,C) \cong \Hom_{D^* \AA}(-,A[k])\mid_\CC$.  Since $\CC$ is a coreflective subcategory, it is closed under retracts (see Lemma \ref{lemma:AdjointRetract}) and hence $(H_l C)[l] \in \CC$ for each $l \in \bZ$.  We have
$$\Hom_\CC((H_l C)[l],C) \cong \Hom_{D^* \AA}((H_l C)[l],A[k]) \cong \Ext^{k-l}(C,A).$$
Note that the left hand side is zero if and only in $(H_l C)[l]$ is zero (since $(H_l C)[l]$ is a direct summand of $C$).  Since $\AA$ is hereditary, the right hand side can only be nonzero for $l=k,k-1$.  This shows that the left hand side is zero when $l \not= k,k-1$ and thus that $(H_l C)[l] = 0$.
\end{proof}

\begin{proposition}\label{proposition:RestrictingAdjoints}
Let $\AA$ be a hereditary category.  If $\CC$ is a (co)reflective subcategory of $D^* \AA$, then $H_k \CC$ is a (co)reflective subcategory of $\AA$.
\end{proposition}

\begin{proof}
We will only show this statement for coreflective subcategories; the other statement is dual.  Note that by Lemma \ref{lemma:AdjointRetract}, the category $\CC$ is closed under retracts so that for every $X \in H_k \CC \subseteq \AA$, we have $X[k] \in \CC$.

Let $A \in \AA$, and consider the object $B = (A[k])_\CC$.  By Lemma \ref{lemma:HereditaryAdjoints}, the homologies of $B$ can only be nonzero in degrees $k,k-1$ (see Lemma \ref{lemma:HereditaryAdjoints}).  In particular, for every $X \in H_k \CC$ we have
\begin{eqnarray*}
\Hom_\AA(X,A) &\cong& \Hom_{D^* \AA} (X[k],A[k]) \\
&\cong& \Hom_\CC(X[k],B) \\
&\cong& \Hom_\CC(X[k],B_{k-1}[k-1]) \oplus \Hom_\CC(X[k], B_k[k]) \\
&\cong& \Hom_\CC(X[k], B_k[k]) \cong \Hom_\AA(X,B_k)
\end{eqnarray*}
natural in the first component.  This shows that $\Hom(-,A)\mid_{H_k \CC} \cong \Hom(-,B_k)$ and thus $H_k \CC$ is a coreflective subcategory of $\AA$.
\end{proof}

\begin{remark}
The converse statement does not holds, i.e. if $\CC$ is a full subcategory of $D^* \AA$ such that $H_k \CC \to \AA$ has a right adjoint for all $k \in \bZ$, then $\CC \to D^* \AA$ does not necessarily have a right adjoint.  We refer to Example \ref{example:NotEnoughInjectives}.
\end{remark}

\subsection{Wide and thick subcategories}

Let $\AA$ be an abelian subcategory.  A \emph{wide} subcategory $\WW \subseteq \AA$ is a full subcategory closed under kernels, cokernels, and extensions.  For a full subcategory $\BB \subseteq \AA$, we will write $\wide_\AA(\BB)$ for the wide closure of $\BB$, that is for the intersection of all wide subcategories of $\AA$ containing $\BB$.  When the ambient category is understood, we will also write $\wide \BB$ instead of $\wide_\AA(\BB)$.  Likewise, for $E \in \AA$ we will write $\wide_\AA(E)$ or $\wide E$ for the wide closure of $E$.

Let $\TT$ be a triangulated subcategory.  A full triangulated subcategory is called thick if it is closed under direct summands.  Given a full subcategory $\SS \subseteq \TT$, we will write $\thick_\TT(\SS)$ for the thick closure of $\SS$ in $\TT$, thus $\thick_\TT(\SS)$ is the intersection of all thick subcategories of $\TT$ containing $\SS$.  Again, when the ambient category $\TT$ is understood, we will also write $\thick \SS$ instead of $\thick_\TT(\SS)$.  Similarly, for $T \in \TT$ we denote by $\thick_\TT(T)$ or $\thick T$ the thick closure of $T$.

For any abelian category $\AA$ and any full $\BB \subseteq \AA$, we have that $\wide \BB$ is an abelian subcategory of $\AA$.  When $\AA$ is hereditary, the embedding lifts to a full embedding $\Db \wide \BB \to \Db \AA$ in the obvious way.  The essential image is the full subcategory $\Db_{\wide \BB} \AA \subseteq \Db \AA$ of all complexes whose homologies are in $\wide \BB$.

Note that $\thick_{\Db \AA}(\BB[0]) \cong \Db \wide_\AA(\BB)$ and that $H_k (\thick_{\Db \AA} (\BB[0])) \cong \wide_\AA(\BB)$.

\subsection{Perpendicular subcategories}\label{subsection:Perpendicular}

Let $\DD$ be a triangulated category, and let $\BB \subseteq \DD$ be a full subcategory.  We will write $\BB^{\perp}$ or ${}^\perp \BB$ for the full subcategory given by all $D \in \DD$ such that $\Hom(\BB,D[n]) = 0$ or $\Hom(D[n],\BB) = 0$ for all $n \in \bZ$, respectively.

Likewise, we write $\BB^{\perp_n}$ for the full subcategory given by all $D \in \DD$ such that $\Hom(B,D[n]) = 0$ for all $B \in \BB$.  Thus $\Ob \BB^{\perp_n} = \{D \in \DD \mid \Hom(\BB,D[n]) = 0\}$ and $\BB^\perp = \cap_{n \in \bZ} \BB^{\perp_n}$.

\begin{remark}
In some of our references, the notation $\BB^\perp$ has been used to describe a different subcategory.
\end{remark}

We will use the following proposition (see \cite[Lemma 3.1]{Bondal89}, also \cite[Lemma 3.1]{VandenBergh00}).

\begin{proposition}\label{proposition:Perpendicular}
Let $\DD$ be a triangulated category, and let $\BB, \CC$ be full triangulated subcategories of $\DD$ such that $\CC \subseteq \BB^\perp$.  The following are equivalent.
\begin{enumerate}
\item $\BB$ and $\CC$ generate $\DD$ as a triangulated category.
\item For each $D \in \DD$ there is a triangle $B \to D \to C \to B[1]$ with $B \in \BB$ and $C \in \CC$.
\item $\CC = \BB^{\perp}$ and the inclusion $i:\BB \to \DD$ has a right adjoint.
\item $\BB = {}^\perp \CC$ and the inclusion $j:\CC \to \DD$ has a left adjoint.
\end{enumerate}
If one of these conditions hold, then the triangles in 2. are unique up to unique isomorphism, and are of the form
$$D_\BB \to D \to D^\CC \to D_\BB[1]$$
where the maps $D_\BB \to D$ and $D \to D^\CC$ are the universal maps.
\end{proposition}

We have similar definitions for an abelian category $\AA$.  Let $\BB \subseteq \AA$ be a full subcategory.  We will write $\BB^{\perp}$ or ${}^\perp \BB$ for the full subcategory given by all $A \in \AA$ such that $\Ext^n (B,A) = 0$ or $\Ext^n (A,B) = 0$ for all $B \in \BB,n \in \bN$, respectively.

If $\AA$ is a hereditary category, then both $\BB^\perp$ and ${}^\perp \BB$ are closed under extensions, kernels and cokernels and are hence also hereditary categories (\cite[Proposition 1.1]{GeigleLenzing91}).

The following proposition is well-known, and gives a connection between the perpendicular categories in the hereditary setting and in the derived (triangulated) setting.

\begin{proposition}
Let $\BB$ be an abelian subcategory of a hereditary abelian category $\AA$, and let $\Db \BB \to \Db \AA$ be the derived functor of this embedding.  We have $\Db (\BB^\perp) = (\Db \BB)^\perp$ and $\Db ({}^\perp \BB) = {}^\perp (\Db \BB)$
\end{proposition}

\begin{proof}
An object $X \in \Db \AA$ will lie in $\Db \BB$ if and only if $H_k X \in \BB$ for all $k \in \bZ$.  Thus
\begin{eqnarray*}
X \in (\Db \BB)^\perp &\Longleftrightarrow& \forall B \in \Db \BB: \Hom_{\Db \AA}(B,X) =0 \\
&\Longleftrightarrow& \forall k \in \bZ, \forall n \in \bN, \forall B \in \Db \BB: \Hom_{\Db \AA}(H_n B[n],H_k X[k]) = 0 \\
&\Longleftrightarrow& \forall k \in \bZ, \forall n \in \bZ, \forall B' \in \BB: \Hom_{\Db \AA}(B'[0],(H_k X)[n]) = 0 \\
&\Longleftrightarrow& \forall k \in \bZ, \forall n \in \bZ, \forall B' \in \BB: \Ext^n_{\AA}(B',H_k X) = 0 \\
&\Longleftrightarrow& \forall k \in \bZ: H_k X \in \Db (\BB^\perp) \\ 
&\Longleftrightarrow& X \in \Db (\BB^\perp).
\end{eqnarray*}
The other statement is proved similarly.
\end{proof}

\subsection{\texorpdfstring{$t$-Structures}{t-Structures}}\label{subsection:tStructures}

Let $\DD$ be any triangulated category.  A $t$-structure (\cite{Beilinson82}) on $\DD$ consists of a pair $(D^{\geq 0}, D^{\leq 0})$ of full subcategories of $\CC$ satisfying the following conditions, where we denote $D^{\leq n} = D^{\leq 0} [-n]$ and $D^{\geq n} = D^{\geq 0} [-n]$
\begin{enumerate}
\item $D^{\leq 0} \subseteq D^{\leq 1}$ and $D^{\geq 1} \subseteq D^{\geq 0}$,
\item $\Hom(D^{\leq 0}, D^{\geq 1}) = 0$,
\item \label{enumerate:Triangles} $\forall Y \in \CC$, there exists a triangle $X \to Y \to Z \to X[1]$ with $X \in D^{\leq 0}$ and $Z \in D^{\geq 1}$.
\end{enumerate}
Furthermore, we will say the $t$-structure is \emph{bounded} if
$$\bigcup_n D^{\leq n} = \bigcup_n D^{\geq n} = \DD,$$
and we will say the $t$-structure is \emph{nondegenerate} if
$$\bigcap_n D^{\leq n} = \bigcap_n D^{\geq n} = \{0\}.$$

A \emph{preaisle} is a full subcategory $\UU \subseteq \TT$ which is closed under suspensions and extensions.  If the embedding $\UU \to \TT$ has a right adjoint, then $\UU$ is called an \emph{aisle}.  Dually, a reflective subcategory $\UU' \subseteq \TT$ which is closed under desuspensions and extensions is called a \emph{coaisle}.

Aisles were introduced in \cite{KellerVossieck88}.  The connection with $t$-structures is given by the assignment $\UU \to (\UU, \UU[1]^{\perp_0})$ which gives a bijection between the aisles and the $t$-structures on $\DD$, where we recall that $\Ob \UU[1]^{\perp_0}$ is the full subcategory of $\DD$ given by
$$\Ob \UU[1]^{\perp_0} = \{D \in \DD \mid \Hom(\UU[1], D) = 0\}$$
as defined in Section \ref{subsection:Perpendicular}.  In particular, a coaisle is uniquely determined by the corresponding aisle, and vice versa.

We will say a subcategory $\UU$ of $\DD$ is \emph{contravariantly finite} if for every $D \in \DD$ there is a $U \in \UU$ and a map $U \to D$ such that $\Hom_\DD(-,U)\mid_\UU \to \Hom_\DD(-,D)\mid_\UU$ is an epimorphism.  Note that a full coreflective subcategory is always contravariantly finite; we can choose $U = D_\UU$.

The following proposition shows that contravariantly finiteness is what distinguishes aisles and preaisles in our setting.

\begin{proposition}\label{proposition:Neeman}\cite[Proposition 1.4]{Neeman10}
Let $\UU$ be a preaisle in a triangulated category $\DD$ where all idempotents split.  If $\UU$ is closed under retracts and $\UU$ is contravariantly finite in $\DD$, then $\UU$ is an aisle.
\end{proposition}

\begin{remark}
Although the statement of \cite[Proposition 1.4]{Neeman10} requires $\UU$ to be a thick subcategory (thus a preaisle closed under retracts and desuspension), the proof does not use being closed under desuspension.
\end{remark}

The previous proposition allows us to generalize \cite[Proposition 1.4]{KellerVossieck88} (see also \cite[1.4]{Beilinson82}) in Proposition \ref{proposition:Chen} below.  We first recall a definition.

\begin{definition}
Let $\TT$ be a triangulated category and let $\UU$ and $\VV$ be two subcategories of $\TT$.  We will denote by $\UU \ast \VV$ the full subcategory of $\TT$ consisting of the objects $X \in \TT$ occurring in a triangle $U \to X \to V \to U[1]$ with $U \in \UU$ and $V \in \VV$.
\end{definition}

It follows from \cite[1.3.10]{Beilinson82} that the operation $\ast$ is associative.  We will use the following lemma.

\begin{lemma}\label{lemma:PreaislesStar}
Let $\UU$ and $\VV$ be preaisles in any triangulated category $\TT$.  If $\UU \subseteq {}^\perp \VV$, then $\UU \ast \VV$ is a preaisle.
\end{lemma}

\begin{proof}
It is clear that $\UU \ast \VV$ is closed under suspension.  To show it is closed under extensions, let $X,Y \in \UU \ast \VV$ and let $X \to M \to Y \to X[1]$ be a triangle.  We want to show that $M \in \UU \ast \VV$.

There are triangles $U_X \to X \to V_X \to U_X[1]$ and $U_Y \to Y \to V_Y \to U_Y[1]$ with $U_X,U_Y \in \UU$ and $V_X,V_Y \in \VV$.  The composition $U_Y \to Y \to X[1] \to V_X[1]$ is zero (since $\UU \subseteq {}^\perp \VV$) so that the map $U_Y \to Y \to X[1]$ factors as $U_Y \to U_X[1] \to X[1]$.

Using the octahedral axiom, we may extend this commutative square to a commutative diagram
$$\xymatrix{
U_X \ar[r] \ar[d] & X \ar[r] \ar[d] & V_X \ar[r]\ar[d] & U_X[1] \ar[d] \\
U_M \ar[r] \ar[d] & M \ar[r] \ar[d] & V_M \ar[r] \ar[d] &U_M[A] \ar[d]\\
U_Y \ar[r] \ar[d] & Y \ar[r] \ar[d] & V_Y \ar[r]\ar[d] & U_Y[1] \ar[d] \\
U_X[1] \ar[r] & X[1] \ar[r] & V_X[1] \ar[r] & U_X[2]}$$
where the rows and columns are triangles.  Since $\UU$ and $\VV$ are preaisles, we have $U_M \in \UU$ and $V_M \in \VV$, and hence $M \in \UU \ast \VV$ as required.
\end{proof}

\begin{proposition}\label{proposition:Chen}
Let $\UU$ and $\VV$ be a aisles in $\Db \AA$.  If $\UU \ast \VV$ is a preailse, then it is an aisle.  In particular if $\UU \subseteq {}^\perp \VV$, then $\UU \ast \VV$ is an aisle.
\end{proposition}

\begin{proof}
Since $\UU$ and $\VV$ are aisles in $\Db \AA$, they are contravariantly finite.  It follows from \cite[Theorem 1.3]{Chen09} that $\UU \ast \VV$ is also contravariantly finite.  By assumption $\UU \ast \VV$ is a preailse.

If $\UU \ast \VV$ were closed under direct summands (= retracts) then, since all idempotents split in $\Db \AA$, we can apply Proposition \ref{proposition:Neeman} to see that $\UU \ast \VV$ is an aisle.  To check that it is indeed closed under direct summands, we will first show that $\UU \ast \VV = \UU \ast (\VV \cap \UU^{\perp_0})$.  It is easy to see that $\UU \ast \VV \supseteq \UU \ast (\VV \cap \UU^{\perp_0})$.  To check the other direction, let $X \in \UU \ast \VV.$

The cone $C$ on the universal map $X_\UU \to X$ lies in $\UU \ast \VV$ since we have assumed the latter is a preaisle, hence there is a triangle
$$U \to C \to V \to U[1]$$
where $U \in \UU$ and $V \in \VV$.  Since $(\UU, \UU[1]^{\perp_0})$ is a $t$-structure, we know that $C \in \UU^{\perp_0}$ so that $U \to C$ is the zero map, and hence $C$ is a direct summand of $V$.  We conclude that $C \in \VV \cap \UU^{\perp_0}$ and hence $X \in \UU \ast (\VV \cap \UU^{\perp_0})$.  This proves the other inequality.

We will now show that $\UU \ast \VV$ is closed under direct summands.  Let $\VV' = \VV \cap \UU^{\perp_0}$ and let $X \in \UU \ast \VV = \UU \ast \VV'$.  For a direct summand $Y$ of $X$ we have the solid arrow diagram
$$\xymatrix{Y_\UU \ar[r] \ar@{-->}[d] & Y \ar[r] \ar[d] & C \ar[r] & Y_\UU[1]\\
U \ar[r] \ar@{-->}[d] & X \ar[r] \ar[d] & V' \ar[r] & U[1] \\
Y_\UU \ar[r] & Y \ar[r] & C \ar[r] & Y_\UU[1]}$$
where $U \in \UU$ and $V' \in \VV'$, and where all the rows are triangles.  The maps $Y \to X$ and $X \to Y$ are a split monomorphism and a split epimorphism, respectively, such that $Y \to X \to Y$ is the identity.  The map $Y_\UU \to U$ is exists since $Y_\UU \to Y \to X \to V'$ is zero (since $Y_\UU \in \UU$), and the map $\UU \to Y_\UU$ is induced by the universal property of the map $Y_\UU \to Y$.  By construction, the dashed arrows make the diagram commute, so there exist maps $C \to V'$ and $V' \to C$ giving maps between the triangles in the last diagram.

Again using the universal property of $Y_\UU \to Y$ we see that $Y_\UU \to U \to Y_\UU$ is the identity, and thus the composition $C \to V' \to C$ is an isomorphism.  We see that $C$ is a direct summand of $V' \in \VV$ and hence $C \in \VV$.  We may conclude that $Y \in \UU \ast \VV$ which finishes the proof.
\end{proof}

\subsection{Torsion classes and nullity classes}\label{subsection:Torsion}

Let $\AA$ be any abelian category.  A \emph{torsion theory} on $\AA$ is a pair $(\TT, \FF)$ of full subcategories of $\AA$ so that $\Hom_\AA(\TT,\FF) = 0$ and for every object $A \in \AA$ there is a short exact sequence
$$0 \to T \to A \to F \to 0$$
where $T \in \TT$ and $F \in \FF$.  This short exact sequence is necessarily unique up to isomorphism.  The objects in $\TT$ are called \emph{torsion} objects and the objects in $\FF$ are called \emph{torsionfree} objects.

The subcategory $\TT$ is called a torsion subcategory of $\AA$.  Any full subcategory of $\AA$ satisfying the following properties is a torsion subcategory:
\begin{enumerate}
\item $\TT$ is closed under quotient objects,
\item $\TT$ is closed under extensions, and
\item the embedding $\TT \to \AA$ has a right adjoint.
\end{enumerate}
The associated torsionfree subcategory is then given by
$$\FF = \TT^{\perp_0} = \{A \in \AA \mid \Hom(\TT,A) = 0\}.$$
A torsion subcategory $\TT$ of $\AA$ yields a corresponding aisle $\UU$ in $\Db \AA$ given by (\cite[Proposition 2.1]{Happel96})
$$\UU = \{X \in \Db \AA \mid \mbox{$H_n X =0$ when $n < 0$, and $H_0 X \in \TT$}\}.$$

Following \cite{StanleyWang11} we will say that a full subcategory $\TT$ of $\AA$ is a \emph{nullity class} if it is closed under quotient objects and under extensions, but without the condition that the embedding $\TT \to \AA$ has a right adjoint (thus a torsion class is a coreflective nullity class).  A nullity class gives a preaisle in the same way as a torsion class gives an aisle.

A nullity class $\TT$ of $\AA$ will be called \emph{tilting} if for every $A \in \AA$ there is a monomorphism $A \to T$ for some $T \in \TT$.
%\section{The wide envelope of a narrow subcategory}

\section{Narrow subcategories and their wide closure}

In this section, let $\AA$ be a hereditary category.  In this section, we give the definition of a narrow subcategory of $\AA$ (Definition \ref{definition:Narrow}) as in \cite{StanleyWang11}.  Our main result is Corollary \ref{corollary:NarrowIsPretorsion} where it is shown that every narrow subcategory $\NN$ in $\AA$ is a nullity class in its wide closure $\wide \NN$.

\begin{definition}\label{definition:Narrow}
A full subcategory $\NN \subset \mathcal A$ is called a \emph{narrow subcategory} if $\NN$ is closed under extensions and cokernels.
\end{definition}

\begin{proposition}\label{proposition:Images}
A narrow subcategory $\NN$ of a hereditary category $\AA$ is closed under images.
\end{proposition}

\begin{proof}
Let $f: A \to B$ be a nonzero map in $\NN$.  Since $\AA$ is hereditary, the exact sequence $0 \to \ker f \to A \to B \to \coker f \to 0$ splits, thus there is an object $J \in \Ob \AA$ and a short exact sequence $0 \to A \to \im f \oplus J \to B \to 0$.  We know that $\CC$ is closed under direct summands (it is closed under quotient objects) and under extensions, hence $\im f \in \Ob \NN$.
\end{proof}

We wish to show that $\NN$ is a tilting nullity class in $\WW$.  We start with the following lemma.

\begin{lemma}\label{lemma:CokernelInNN}
Let $\NN$ be a narrow subcategory in an abelian hereditary category $\AA$.  Let $X \to Y$ be a map between objects in $\NN$, and denote the kernel by $K$.  If $X' \in \NN$, then the cokernel $C$ of any map $K \to X'$ lies in $\NN$.
\end{lemma}

\begin{proof}
First, we recall from Proposition \ref{proposition:Images} that $\NN$ is closed under images of maps.  So without loss of generality, we may thus assume that the map $X \to Y$ is an epimorphism.

Consider the following commutative diagram with exact rows and columns, where the upper left square is a pushout.
$$\xymatrix{0 \ar[r] & K \ar[r] \ar[d] & X \ar[r] \ar[d]& Y \ar[r] \ar@{=}[d]& 0 \\
0 \ar[r] & X' \ar[r] \ar[d]& P \ar[r] \ar[d]& Y \ar[r] & 0 \\
& Y' \ar[d]\ar@{=}[r] & Y'\ar[d] \\
& 0 & 0}$$
Since $\NN$ is closed under extensions, the second row implies that $P \in \NN$.  The second column then shows that $Y' \in \NN$, as required.
\end{proof}

\begin{proposition}\label{proposition:GeneratedInOneStep}
Let $\AA$ be a hereditary abelian category and let $\NN \subseteq \AA$ be a narrow subcategory.  Every object in $\WW = \wide \NN$ is the kernel of a map in $\NN$.
\end{proposition}

\begin{proof}
Let $\KK$ be the full subcategory of $\AA$ given by $K \in \KK \Leftrightarrow K \cong \ker(X \to Y)$ with $X,Y \in \NN$.  Note that $\NN \subseteq \KK \subseteq \WW$.  We shall show that $\KK = \WW$ by showing that $\KK$ is a wide subcategory of $\AA$.  Since $\NN$ is closed under images (Proposition \ref{proposition:Images}), we may assume that the aforementioned map $X \to Y$ is an epimorphism.  Thus assume that there are short exact sequences $0 \to A \to X \to Y \to 0$ and $0 \to B \to X' \to Y' \to 0$ where $X,Y,X',Y' \in \NN$.

First we will show that $\KK$ is closed under extensions.  Thus any $C \in \AA$ occurring in a short exact sequence $0 \to A \to C \to B \to 0$ is a kernel of a map in $\NN$.

We will show that there is a commutative diagram with exact rows
$$\xymatrix{
0 \ar[r] & A \ar[r] \ar[d] & C \ar[r] \ar[d] & B \ar[r] \ar@{=}[d] & 0 \\
0 \ar[r] & X \ar[r] \ar@{=}[d] & C' \ar[r] \ar[d] & B \ar[r] \ar[d] & 0 \\
0 \ar[r] & X \ar[r] & Z \ar[r] & X' \ar[r] & 0}$$
where the upper left square and the lower right square are both pushouts and pullbacks, and where the vertical maps are monomorphisms.  The middle exact sequence is constructed by taking the pushout of the span $X \leftarrow A \rightarrow C$.  For bottom lower exact sequence, one uses heredity to see that $\Ext(X',X) \to \Ext(B,X)$ is surjective; the bottom exact sequence is then an element from $\Ext(X',X)$ which gets mapped to the middle exact sequence.  Thus the lower right square is a pullback.  The diagram shows that $C$ is a subobject of $Z \in \NN$.  Using the Snake Lemma we may complete to the commutative diagram
$$\xymatrix{
&0\ar[d]&0\ar[d]&0\ar[d] \\
0 \ar[r] & A \ar[r] \ar[d] & C \ar[r] \ar[d] & B \ar[r] \ar[d] & 0 \\
0 \ar[r] & X \ar[r] \ar[d] & Z \ar[r] \ar[d] & X' \ar[r] \ar[d] & 0 \\
0 \ar[r] & Y \ar[r] \ar[d] & Z' \ar[r] \ar[d]& Y' \ar[r] \ar[d] & 0 \\
&0&0&0}$$
where the rows and the columns are exact.  This shows that $C$ is the kernel of the map $Z \to Z'$ in $\NN$. 

Next, assume that $C$ is a kernel of a map $A \to B$, where $A,B \in \KK$ as above.   Note that since $\KK$ is closed under direct summands (this follows from Lemma \ref{lemma:CokernelInNN} where the map $K \to X'$ is given by $K \stackrel{e}{\rightarrow} K \to X$ where $e \in \End(K)$ is an idempotent) and under extensions, it follows that $\KK$ is closed under images (as in the proof of Proposition \ref{proposition:Images}).  We wish to show that $C \in \KK$.  Since $\KK$ is closed under images, we may assume that $A \to B$ is an epimorphism.  There is a commutative diagram
$$\xymatrix{
&0\ar[d]&0\ar[d]&0\ar[d] \\
0 \ar[r] & C \ar[r] \ar[d] & A \ar[r] \ar[d] & B \ar[r] \ar[d] & 0 \\
0 \ar[r] & X \ar[r] \ar[d] & X \oplus X' \ar[r] \ar[d] & X' \ar[r] \ar[d] & 0 \\
0 \ar[r] & Z \ar[r] \ar[d] & Q \ar[r] \ar[d]& Y' \ar[r] \ar[d] & 0 \\
&0&0&0}$$
with exact rows and columns.  Here the map $A \to X \oplus X'$ is given by the monomorphism $A \to X$ in the first component and the map $A \to B \to X'$ in the second.  By Lemma \ref{lemma:CokernelInNN} we know that $Q \in \NN$.  Note that $C$ is the kernel of the map $X \to Z \to Q$ as required.

Finally, Let $C$ be a cokernel of a map $A \to B$.  As before, we may assume that this map is a monomorphism.  We now find a commutative diagram
$$\xymatrix{
&0\ar[d]&0\ar[d] \\
0 \ar[r] & A \ar[r] \ar[d] & B \ar[r] \ar[d] & C \ar[r] \ar[d] & 0 \\
0 \ar[r] & X' \ar@{=}[r] \ar[d] & X' \ar[r] \ar[d] & 0 \ar[r] \ar[d] & 0 \\
& Z \ar[r] \ar[d] & Y' \ar[r] \ar[d]& 0 \ar[r] \ar[d] & 0 \\
&0&0&0}$$
with exact rows and columns.  By Lemma \ref{lemma:CokernelInNN} we know that $Z \in \NN$.  The Snake Lemma now implies that $C$ is the kernel of the map $Z \to Y'$ in $\NN$.  This concludes the proof.
\end{proof}

\begin{corollary}\label{corollary:NarrowIsPretorsion}
A full subcategory $\NN \subseteq \AA$ is a narrow subcategory if and only if $\NN$ is a tilting nullity class in its wide closure.
\end{corollary}

\begin{proof}
It follows from Propositions \ref{proposition:Images} and \ref{proposition:GeneratedInOneStep} that a narrow subcategory is indeed a tilting nullity class in its wide closure.  The other direction is straightforward.
\end{proof}

\begin{corollary}\label{corollary:NarrowIsTorsion}
Let $\NN$ be a coreflective narrow subcategory of $\AA$, then $\NN$ is a tilting torsion subcategory of $\wide \NN$.
\end{corollary}

\begin{example}
When $\AA$ is the category of finitely generated modules over a commutative Noetherian ring, then the narrow subcategories have been classified in \cite[Theorem 4]{StanleyWang11}.  In this case, a narrow subcategory is equal to its wide closure.
\end{example}

\begin{remark}
We refer the reader to Proposition \ref{proposition:ProjectiveLineNarrow} below for a classification of narrow subcategories on the category $\coh \bP^1$ of coherent sheaves on a smooth projective line.
\end{remark}

\section{Narrow sequences and preaisles}\label{Section:NarrowSequence}

Recall from Definition \ref{definition:Narrow} that a narrow subcategory of $\AA$ is a subcategory which is closed under extensions and cokernels.  In this section, we shall consider narrow sequences (Definition \ref{definition:NarrowSequence}) and consider the connection with preaisles (Theorem \ref{theorem:NarrowPreaisle}).

%In this section, let $\AA$ be a hereditary category.

\begin{definition}\label{definition:NarrowSequence}
A sequence $(\NN(k))_{k\in \bZ}$ of narrow subcategories is called a \emph{narrow sequence} if for all $k \in \bZ$
\begin{enumerate}
\item $\NN(k) \subseteq \NN(k+1)$, and
\item if $A \rightarrow B \rightarrow C \rightarrow D \rightarrow E$ is exact with $A \in \NN(k+1)$ and $B, D \in \NN(k)$ and $E \in \NN(k-1)$, then $C \in \NN(k)$.
\end{enumerate}
\end{definition}

\begin{lemma}\label{lemma:OnlySequence}
If $(\CC(k))_{k \in \bZ}$ is a sequence of full subcategories (each containing $0$) satisfying the conditions (1) and (2) of Definition \ref{definition:NarrowSequence}, then each $\CC(k)$ is automatically a narrow subcategory.
\end{lemma}

\begin{proof}
That $\CC(k)$ is closed under extensions follows from taking $A \cong 0 \cong E$ in the second condition.  To show that is it closed under cokernels, let $f:X \to Y$ be a morphism in $\CC(k)$.  Since $X \in \CC(k) \subseteq \CC(k+1)$, the second condition implies $\coker f \in \CC(k)$ as required.
\end{proof}

\begin{proposition}\label{proposition:KernelInNext}
Let $(\NN(k))_k$ be a narrow sequence.  If $f:D \to E$ is a map between objects in $\NN(k)$, then $\ker f \in \NN(k+1)$.
\end{proposition}

\begin{proof}
This follows from the short exact sequence $0 \to \ker f \to D \to E$ with $D \in \NN(k) \subseteq \NN(k+1)$ so that $\ker f \in \NN(k+1)$.
\end{proof}

\begin{corollary}\label{corollary:GrowingFastEnough}
If $\AA$ is a hereditary category and $(\NN(k))_k$ is a narrow sequence in $\AA$, then $\wide (\NN(k)) \subseteq \NN(k+1)$.
\end{corollary}

\begin{proof}
This follows from Propositions \ref{proposition:GeneratedInOneStep} and \ref{proposition:KernelInNext}.
\end{proof}

\begin{definition}\label{definition:Mu}
Let $\UU$ be a preaisle.  We define a sequence $\mu(\UU) = \{\CC(k)\}_k$ of subcategories of $\AA$ by $\CC(k) = H_{k} \UU$.
\end{definition}

\begin{definition}\label{definition:HomologyDetermined}
Let $\UU$ be a preaisle in $\Db \AA$.  We will say that $\UU$ is \emph{homology-determined} if for all $X \in \Db \AA$ we have
$$X \in \UU \Longleftrightarrow \forall k \in \bZ: (H_k X) [k] \in \UU.$$
\end{definition}

Not all preaisles are homology-determined (we refer to Example \ref{example:BadPreaisles} and \ref{example:Preaisle}).  The following proposition shows that, in case $\AA$ is hereditary, homology-determined preaisles lie between preaisles and aisles.

\begin{proposition}\label{proposition:RetractsCoproducts}
Let $\AA$ be a hereditary category.  Every aisle $\UU \subseteq \Db \AA$ is a homology-determined preaisle.
\end{proposition}

\begin{proof}
Let $\UU$ be an aisle in $\Db \AA$.  If $X \in \UU$ is then by Lemma \ref{lemma:AdjointRetract} we know that $X_k [k] \in \UU$.  Conversely, if $X_k[k] \in \UU$, for all $k \in \bZ$, then also $X \cong \coprod_k (X_k [k]) \in \UU$.
\end{proof}

\begin{proposition}\label{proposition:PreaisleNarrow}
Let $\UU$ be a homology-determined preaisle in $\Db \AA$, where $\AA$ is an abelian category.  The sequence $(\CC(k))_k = \mu (\UU)$ is a narrow sequence in $\AA$.
\end{proposition}

\begin{proof}
By Lemma \ref{lemma:OnlySequence} we only need to show that $(\CC(k))_{k \in \bZ}$ satisfies the conditions (1) and (2) from Definition \ref{definition:NarrowSequence}.  Since $\UU$ is closed under suspension, we see that $\CC(k) \subseteq \CC(k+1)$.

For the second condition, let $A \stackrel{f}{\rightarrow} B \rightarrow C \rightarrow D \stackrel{g}{\rightarrow} E$ be exact with $A \in \CC(k+1)$ and $B, D \in \CC(k)$ and $E \in \CC(k-1)$.  Since $\UU$ is homology-determined, we know that $A[k+1], B[k], D[k], E[k-1] \in \UU$.  Since preaisles are closed under extensions, we know that $\UU$ contains both $\cone(f[k]:A[k] \to B[k])$ and $\cone(g[k-1]:D[k-1] \to E[k-1])$.

Thus $\coker f \cong H_k(\cone (f[k])) \in \CC(k)$ and $\ker g \cong H_k(\cone (g[k-1])) \in \CC(k)$.  It is clear that $\CC(k)$ is closed under extensions (since $\UU$ is homology-determined and closed under extensions), hence $C \in \CC(k)$.
\end{proof}

\begin{definition}\label{definition:Theta}
For a narrow sequence $(\NN(k))_k$, define 
$$\theta ((\NN(k))_k) = \{ A\in \Db {\AA} | \mbox{$H_{k}(A)\in \NN(k)$ for  all $k \in \bZ$}\}.$$ 
\end{definition}

\begin{proposition}\label{proposition:NarrowPreaisle}
For a narrow sequence $(\NN(k))_{k \in \bZ}$, the category $\theta ((\NN(k))_{k\in \bZ}) \subseteq \Db \AA$ is a homology-determined preaisle.
\end{proposition}

\begin{proof}
Since $\NN(k) \subseteq \NN(k+1)$, we see that $\theta ((\NN(k))_k)$ is closed under suspensions.

Let $X \to Y \to Z \to X[1]$ be a triangle in $\Db \AA$ where $X,Z \in \theta ((\NN(k))_k)$.  From this triangle, we get the long exact sequence
$$\cdots \to H_{k+1} Z \to H_{k} X \to H_{k} Y \to H_{k} Z \to H_{k-1} X \to \cdots$$
in $\AA$.  We conclude that $Y \in \theta ((\NN(k))_k)$.

Furthermore, it follows directly from the definitions that $\theta ((\NN(k))_k)$ is homology-determined.
\end{proof}

The following theorem is the main result of this section.  It relates narrow sequences, a structure in the abelian category, to homology-determined preaisles, a structure in the derived category.

\begin{theorem}\label{theorem:NarrowPreaisle}
Let $\AA$ be an abelian category.  The functions $\mu$ from Definition \ref{definition:Mu} and $\theta$ from Definition \ref{definition:Theta} give inverse bijections
$$\begin{array}{lcr}
\left\{
\mbox{Narrow sequences in $\AA$}
\right\}&
\longleftrightarrow&
\left\{ \mbox{Homology-determined preaisles in $\Db \AA$} \right\}.
\end{array}$$
\end{theorem}

\begin{proof}
Propositions \ref{proposition:PreaisleNarrow} and \ref{proposition:NarrowPreaisle} yield that $\theta((\NN(k))_k)$ is indeed a homology-determined preaisle as in the statement of the theorem and that $\mu(\UU)$ is a narrow sequence.

For a narrow sequence $(\NN(k))_k$, we have $(\NN(k))_k = \mu(\theta((\NN(k))_k))$ by definition of $\mu$ and $\theta$.  Likewise, for a homology-determined preaisle $\UU$, we have $\UU = \theta (\mu(\UU))$.  This shows that $\mu$ and $\theta$ are inverse bijections.
\end{proof}

\begin{example}\label{example:BadPreaisles}
The full subcategory $\UU$ of $\Db \mod k$ generated by all objects $X$ with $H_k X = 0$ when $k < 0$ and $\sum_k \dim H_k X$ even is a preaisle.  We have however
$$H_k \UU \cong \left\{\begin{array}{ll} 0 & k < 0 \\ \mod k & k \geq 0 \end{array} \right.$$
so that $\theta (\{H_k \UU\}_k)$ is the standard aisle.  Note that $\UU$ is not homology-determined.
\end{example}

\begin{example}\label{example:Preaisle}
Consider the quiver $Q:1 \stackrel{\alpha}{\rightarrow} 2 \stackrel{\beta}{\rightarrow} 3$ with relation $\beta\alpha =0$.  Let $P_1,P_2$ be indecomposable projective representations associated with the vertices $1$ and $2$ respectively, and let $f:P_2 \to P_1$ be a nonzero map.  It is straightforward to check that
$$\cdots \longrightarrow 0 \longrightarrow P_2 \stackrel{f}{\longrightarrow} P_1 \longrightarrow 0 \longrightarrow \cdots$$
represents a complex $X \in \Db \rep Q \cong \Kb \proj (\rep Q)$ with $\End^\bullet (X) \cong k$ (thus $X$ is an exceptional object), so that the full additive subcategory $\UU$ of $\Db \rep Q$ spanned by $X[n]$ (for all $n \in \bZ$) is an aisle in $\Db \rep Q$ (this can either be checked directly, or one can use that $\UU \cong \Db \mod k$ such that the right adjoint to $\UU \to \Db \rep Q$ follows from \cite{BondalKapranov89}).

We have constructed an aisle $\UU$ in $\Db \rep Q$ which is not homology-determined (since $H_0(\UU)$ contains the simple modules $S_1$ and $S_3$ associated to the vertices $1$ and $3$).
\end{example}
\section{\texorpdfstring{$t$-Structures for a weighted projective line}{t-Structures for a weighted projective line}}\label{section:ProjectiveLine}

In this section, we will determine the set of $t$-structures on $\coh \bP^1$, the category of coherent sheaves on $\bP^1$, using techniques from this paper.  Let $\OO$ be the structure sheaf and $\OO(n)$ be the $n^{\text {th}}$ Serre twist of $\OO$ (where $n \in \bZ$).  For a closed point $P \in \bP^1$, let $k(P)$ be a simple sheaf supported on $P$.  We will also use the short exact sequence
$$0 \to \OO(n) \to \OO(n+1) \to k(P) \to 0$$
which exists for all $n \in \bZ$ and all closed $P \in \bP^1$.

As a starting point, we will consider the wide subcategories of $\coh \bP^1$; these are well-known.  There are four types.

\begin{proposition}\label{proposition:ProjectiveLineWide}
All wide subcategories of $\coh \bP^1$ occur in the following list.
\begin{enumerate}
\item The zero subcategory is wide.
\item For each subset $P \subseteq \bP^1_k$, the set of torsion sheaves supported on $P$ forms a wide subcategory.
\item For each $n \in \bZ$, the full additive subcategory generated by $\OO(n)$ is wide in $\coh \bP^1$.
\item The category $\coh \bP^1$ is a wide subcategory of itself.
\end{enumerate}
\end{proposition}

\begin{proof}
The claim is straightforward to check.  Alternatively, one can use \cite[Example 5.7]{Bruning07} to see that the lattice of wide subcategories of $\coh \bP^1$ corresponds to the lattice of wide subcategories of $\mod K$ where $K$ is the Kronecker algebra; the lattice of wide subcategories in the latter category has been classified in \cite[Proposition 6.13]{Krause11}, see also \cite[Theorem 3.2.15]{Dichev09}
\end{proof}

By Corollary \ref{corollary:NarrowIsPretorsion}, we find all narrow subcategories of $\coh \bP^1$ as tilting nullity classes in the above wide subcategories.

For the first three types of the wide subcategories, there is only one choice of a tilting nullity class, which is then given by the wide subcategory itself.  Thus to understand the narrow subcategories in $\coh \bP^1$, we need only to consider the tilting nullity classes $\TT$ in $\coh \bP^1$.

It is clear that $\TT$ needs to contain all the torsion sheaves.  Furthermore, since there is no monomorphism $\OO \to T$ for any torsion sheaf $T$, we know that $\TT$ has to contain at least one line bundle $\OO(n)$.  For any closed $p \in \bP^1$, let $k(p) \in \coh \bP^1$ be the simple sheaf supported on $p$.  Since there is a short exact sequence $0 \to \OO(n) \to \OO(n+1) \to k(p) \to 0$, we know that $\OO(m) \in \TT$, for all $m \geq n$.  This means that either there is a minimal $n \in \bZ$ such that $\OO(n) \in \TT$, or all vector bundles lie in $\TT$.  Note that in the former case, every object $X \in \TT$ is a quotient of $\OO(n)^{\oplus l}$, where $l \in \bN$ depends on $X$.  We will denote this subcategory by $\Gen(\OO(n))$.  Thus we have the following five types of narrow subcategories of $\coh \bP^1$.

\begin{proposition}\label{proposition:ProjectiveLineNarrow}
The following subcategories of $\coh \bP^1$ are the only narrow subcategories:
\begin{enumerate}
\item the zero subcategory,
\item for each nonempty subset $P \subseteq \bP^1_k$, the set of torsion sheaves supported on $P$,
\item for each $n \in \bZ$, the full additive subcategory generated by $\OO(n)$,
\item for each $n \in \bZ$, the full subcategory $\Gen(\OO(n))$,
\item the category $\coh \bP^1$ is a narrow subcategory of itself.
\end{enumerate}
\end{proposition}

Next we will discuss all narrow sequences in $\coh \bP^1$ and hence (see Theorem \ref{theorem:NarrowPreaisle}) all homology-determined preaisles in $\Db \coh \bP^1$.  We start our discussion with a lemma.

\begin{lemma}\label{lemma:ProjectiveLineNarrowSequence}
Let $(\NN(k))_k$ be a narrow sequence in $\coh \bP^1$.  If $\NN(k)$ contains a nonzero torsion sheaf, and $\NN(k+1)$ contains a nonzero vector bundle, then $\NN(k+1) = \coh \bP^1$. 
\end{lemma}

\begin{proof}
If $\NN(k)$ contains a nonzero torsion sheaf, then it contains a simple torsion sheaf $S$.  For any $\OO(n) \in \NN(k+1)$ we have a short exact sequence
$$0 \to \OO(n-1) \to \OO(n) \to S \to 0$$
so that Definition \ref{definition:NarrowSequence} implies that $\OO(n-1) \in \NN(k+1)$.  We then see that $\OO(m) \in \NN(k+1)$ for all $m \leq n$.  Since $\NN(k+1)$ is closed under extensions and cokernels, this implies that $\NN(k+1) = \coh \bP^1$.
\end{proof}

Thus, the narrow sequences fall into three types (here, the types of narrow subcategories refer to the numbering given in Proposition \ref{proposition:ProjectiveLineNarrow}).

\newcounter{MyCounter}
\begin{list}{\Roman{MyCounter}.}{\usecounter{MyCounter}}
\item Firstly, assume that $(\NN(k))_k$ does contain a narrow subcategory of type (2).  In this case, Lemma \ref{lemma:ProjectiveLineNarrowSequence} shows that the only possibilities for each $(\NN(k))_k$ are of the form (1) - (2) - (5).  Note that $\NN(k)$ does not need to contain all of these types, but as $k$ increases the type cannot go down (this is because $(\NN(k))_k$ is an increasing sequence).  Furthermore, the supports of the narrow subcategories in $\NN(k)$ of type (2) have to be nondecreasing.

Thus such a narrow sequence is given by $-\infty \leq l_1 < l_2 \leq +\infty$ where $l_1,l_2 \in \bZ$, and a nondecreasing set $\{P_k \subseteq \bP^1 \}_{l_1 \leq k < l_2}$ so that
$$\NN(k) = \begin{cases} 
0 & k < l_1 \\
\mbox{torsion sheaves supported at $P_k$} & l_1 \leq k < l_2 \\
\coh \bP^1 & l_2 \leq k
\end{cases}$$
It is straightforward to check that all of these sequences $(\NN(k))_k$ are indeed narrow sequences.

\item Secondly, we will now assume that the narrow sequence $(\NN(k))_k$ does contain a narrow subcategory of type (3).  In this case, the only possibilities for the narrow subcategories are the types (1) - (3) - (4) - (5).  Similarly as before, $(\NN(k))_k$ does not need to contain all of these types, but as $k$ increases the type cannot go down.  We have the following extra conditions:
\begin{itemize}
\item all narrow subcategories of type (3) are the same (and there is at least one),
\item it follows from Proposition \ref{proposition:GeneratedInOneStep} that there is at most one subcategory of type (4),
\item and if there is a subcategory of type (5) then there is exactly one of type (4) which is given by all quotient objects of the narrow subcategory of type (3).  To see this, assume that $\NN(k-1)$ consists of direct sums of $\OO(n)$ and $\NN(k+1) = \coh \bP^1$.  Since $\NN(k-1) \subseteq \NN(k)$, we know that $\OO(n) \in \NN(k)$.  In $\coh \bP^1$, there is s a short exact sequence
$$0 \to \OO(n-1) \to \OO(n) \to S \to 0$$
where $S$ is any simple sheaf in $\coh \bP^1$.  Since $\OO(n-1) \in \NN(k+1)$, the exact sequence in Definition \ref{definition:NarrowSequence} shows that $S \in \NN(k)$.  Thus $\NN(k)$ is either of type (4) or of type (5).  If $\NN(k)$ were of type (5), then the same reasoning would show that $\NN(k-1)$ would be of type (4) or type (5).  Since we have assumed that $\NN(k-1)$ is of type (3), we may conclude that $\NN(k)$ is of type (4).  We then see that $\OO(n) \in \NN(k)$ and $\OO(n-1) \not\in \NN(k)$ so that $\NN(k) = \Gen(\OO(n))$.
\end{itemize}

Here, one can list the narrow sequences by $-\infty \leq l_1 < l_2 \leq +\infty$ where $l_1,l_2 \in \bZ$, and line bundle $\OO(n) \in \coh \bP^1$.  The corresponding narrow sequence is then given by
$$\NN(k) = \begin{cases} 
0 & k < l_1 \\
\langle \OO(n) \rangle & l_1 \leq k < l_2 \\
\Gen(\OO(n)) & k = l_2 \\
\coh \bP^1 & l_2 < k
\end{cases}$$
Again, it is straightforward to check that all of these sequences $(\NN(k))_k$ are indeed narrow sequences.

\item Thirdly, there are those which contain a narrow subcategory of type (4), but none of type (2) or (3); these sequences are thus of type (1) - (4) - (5).  We can describe them explicitly by
$$\NN(k) = \begin{cases} 
0 & k < l \\
\Gen( \OO(n) ) & k=l \\
\coh \bP^1 & l < k
\end{cases}$$
for some $n \in \bZ$ and $l \in \bZ$.

\item Lastly, there are the narrow sequences of type (1) - (5).  These are given by
$$\NN(k) = \begin{cases} 
0 & k < l \\
\coh \bP^1 & l \leq k
\end{cases}$$
where $l \in \bZ \cup \{\pm \infty\}$.  If $l=0$, then this is the standard aisle; if $l = -\infty$ or $l=\infty$, then the associated preaisle is given by $0$ and $\Db \coh \bP^1$, respectively.
%\end{enumerate}
\end{list}
We have shown the following proposition.

\begin{proposition}\label{proposition:ProjectiveLineNarrowSequence}
All narrow sequences are of the four forms described above.
\end{proposition}

Having described the homology-determined preaisles, we will now check which ones of these are aisles.  

\begin{proposition}\label{proposition:AislesProjectiveLine}
All narrow sequences in Proposition \ref{proposition:ProjectiveLineNarrowSequence} come from aisles, except those of type (1) - (2) - (5) where we need the extra condition that there is only one narrow sequence of type (2).
\end{proposition}

\begin{proof}
We start with the list of narrow sequences in Proposition \ref{proposition:ProjectiveLineNarrowSequence}.  Let $(\NN(k))_k$ be a narrow sequence such that $\theta((\NN(k))_k)$ is an aisle.  We have the following observation: if $\NN(k)$ is of type (2), then $\NN(k+1)$ is of type (5).

Indeed, the list from Proposition \ref{proposition:ProjectiveLineNarrowSequence}, we know that $\NN(k+1)$ is of the form (2) or (5).  If it is of the form (5), then we are done.  We may thus assume that $\NN(k)$ is of the form (2), and in particular it is an abelian subcategory.  Also, $\Hom(\NN(k+1),\OO) = 0$ but $\Ext(\NN(k),\OO) \not= 0$.

Let $\UU = \theta((\NN(k))_k)$ be the associated preaisle.  Assume that $\UU$ is an aisle, thus $\UU \subseteq \coh \bP^1$ is coreflective.  This means that the functor $\Hom(-,\OO(m)[k+1])|_\UU$ is representable and thus that the functor $\Ext(-,\OO(m))|_{\NN(k)}$ is representable (note that $\NN(k)$ is abelian).  However, this last functor is nonzero and right exact and thus needs to be represented by a nonzero injective object in $\NN(k)$.  But $\NN(k)$ has no such objects (since $\Ext(X,X) \not=0$ for all objects $X \in \NN(k)$), which is a contradiction.

To prove the proposition, we need to show that the given possibilities do correspond to aisles. For all types, except type II, this is clear as they are either trivial or induced by a torsion theory.  Thus let $(\NN(k))_k$ be a nontrivial narrow sequence of type II and let $\UU = \theta((\NN(k))_k)$ be the associated preaisle.  Using notation as before, define narrow sequences
$$\NN_1(k) = \begin{cases} 
0 & k < l_1 \\
\langle \OO(n) \rangle & l_1 \leq k ,
\end{cases}$$
and
$$\NN_2(k) = \begin{cases} 
0 & k < l_2 \\
\Gen(\OO(n)) & k = l_2 \\
\coh \bP^1 & l_2 < k
\end{cases}$$
where $-\infty \leq l_1 < l_2 < \infty$.  We write $\UU_1 = \theta((\NN_1(k))_k)$ and $\UU_2 = \theta((\NN_2(k))_k)$.  It is clear that $\UU = \UU_1 \ast \UU_2$.  Here, $\UU$ is a preaisle (since it corresponds to a narrow sequence) and $\UU_2$ is an aisle (it is an aisle induced by a torsion theory on $\coh \bP^1$), so we need only to show that $\UU_1$ is an aisle in $\Db \coh \bP^1$ to apply Proposition \ref{proposition:Chen} and conclude that $\UU$ is an aisle.

We can factor the embedding $\UU_1 \to \Db \coh \bP^1$ as $\UU_1 \to \thick \UU_1 \to \Db \coh \bP^1$.  Note that $\UU_1$ is an aisle in $\thick \UU_1 \cong \mod k$ and that $\thick \UU_1 \to \Db \coh \bP^1$ has a right adjoint (given by $X \mapsto \RHom(\OO(n)[0],X) \otimes \OO(n)[0]$ (alternatively one can use that $\thick \UU_1 \cong \mod k$ is saturated in the sense of  \cite{BondalKapranov89}) and hence the composition $\UU_1 \to \thick \UU_1 \to \Db \coh \bP^1$ has a right adjoint.  This finishes the proof.
\end{proof}

\begin{corollary}\label{corollary:AislesProjectiveLine}
The aisles in $\Db \coh \bP^1$ are either
\begin{enumerate}
\item $0$ or $\Db \coh \bP^1$,
\item induced by a torsion theory, or
\item of type II.
\end{enumerate}
\end{corollary}

\begin{remark}
It is well-known that $\coh \bP^1$ is derived equivalent to the category $\rep Q$ of finite dimensional representations of the Kronecker quiver $Q$.  Hence the set of $t$-structures coincides.  In a subsequent paper, we will use Theorem \ref{theorem:Reduced} below to give another description of the set of $t$-structures on $\coh \bP^1$.
\end{remark}
\section{Some operations on aisles}\label{section:Aisles}

In this section, we have gathered the more technical results which will be used in the rest of this paper.  We will show how, in our setting, one can break up aisles and recover the original aisle.  Section \ref{section:CoreflectiveNarrowSequence} below will only use Corollary \ref{corollary:Pasting}.  The rest of the section will be used in Section \ref{section:LeftAdjoint}.

In this section, let $\AA$ be any hereditary abelian category.  All aisles and preaisles will be in $\Db \AA$.

\begin{definition}\label{definition:Restrictions}
Let $\UU$ be a homology-determined preaisle.  Recall that $D^{\leq n} = D^{\leq 0} [-n]$ where $D^{\leq 0}$ is the standard aisle. We define %Let $\TT(l)$ be the thick subcategory of $D^* \AA$ generated by $H_l \UU$.  We define
\begin{eqnarray*}
\UU |_{-\infty}^{+\infty} &=& \UU, \\
\UU |_{k}^{+\infty} &=& \UU \cap D^{\leq -k} \\
\UU |_{-\infty}^l &=& \UU \cap \thick H_l \UU, \\
\UU |_k^l &=& \UU |_{k}^{+\infty} \cap \UU |_{-\infty}^l.
\end{eqnarray*}
\end{definition}

Put differently, we have
\begin{eqnarray*}
\UU |_{k}^{+\infty} &=& \{X \in \UU | \forall i < k: H_i X = 0 \}, \\
\UU |_{-\infty}^l &=& \{X \in \UU | \forall i \in \bZ: H_i X \in \wide \UU_l \}.
\end{eqnarray*}

The following proposition shows that a preaisle which is ``patched together'' from an increasing sequence of subaisles, is an aisle itself.

Let $X \in \Db \AA$.  In what follows, we will write $X_{[a,b]}$ for $\tau^{\geq -b} \tau^{\leq -a} X$.  Note that $H_k (X_{[a,b]}) \not= 0$ implies that $k \in [a,b]$.  Similarly, for any full subcategory $\CC \subseteq \Db \AA$ we will write $\CC_{[a,b]}$ for the essential image of $\CC$ under $\tau^{\geq -b} \tau^{\leq -a}$.

\begin{proposition}\label{proposition:UnionOfAisles}
Let $\AA$ be an abelian category of finite global dimension.  Let $(\UU(n))_{n \in \bZ}$ be a sequence of aisles with the following property: for each $k \in \bZ$, we have $(\UU(n))_{[-k,k]} = (\UU(n+1))_{[-k,k]}$ for all $n \gg 0$.  We have
\begin{enumerate}
\item if the sequence $(\UU(n))_{n \in \bZ}$ is decreasing (thus $\UU(n) \supseteq \UU(n+1)$ for all $n \in \bZ$), then $\bigcap_{n \in \bZ} \UU(n)$ is an aisle, and
\item if the sequence $(\UU(n))_{n \in \bZ}$ is increasing (thus $\UU(n) \subseteq \UU(n+1)$ for all $n \in \bZ$), then $\bigcup_{n \in \bZ} \UU(n)$ is an aisle.
\end{enumerate}
\end{proposition}

\begin{proof}
We will show that $\UU = \bigcup_{n \in \bZ} \UU(n)$ is an aisle when the sequence $\UU(n)$ is increasing.  The other statement is similar.  It is clear that $\bigcup_n \UU(n)$ is a preaisle (here we use that the sequence is increasing).  To show it is an aisle, we need to show that the embedding $\UU \to \Db \AA$ has a right adjoint.  Let $X \in \Db \AA$.  We wish to show that $\Hom(-,X)|_\UU$ is representable.

Let $d$ be the global dimension of $\AA$, and let $l \in \bN$ be such that $X_{[-l,l]} \cong X$.  For any object $Y \in \Db \AA$, we have that $\Hom(Y,X) \cong \Hom(Y_{[-l-d,l]}, X)$.  Let $k \geq l+d$.  Using that $(\UU(n))_{[-k,k]} = (\UU(n+1))_{[-k,k]}$ for $n \gg 0$, one can verify that $X_{\UU(n)} \cong X_{\UU(n+1)}$ for $n \gg 0$.  Hence $\Hom(-,X)|_\UU \cong \Hom(-,X_{\UU(n)})$.
\end{proof}

We are interested in the specific case where $\AA$ is hereditary (thus the global dimension is one).

\begin{corollary}\label{corollary:UnionOfAisles}
Let $\AA$ be a hereditary category, and let $(\UU(n))_{n \in \bZ}$ be a sequence of aisles in $\Db \AA$ with the following property: for each $k \in \bZ$, we have $H_k \UU(n) = H_k \UU(n+1)$ for all $n \gg 0$.  We have
\begin{enumerate}
\item if the sequence $(\UU(n))_{n \in \bZ}$ is decreasing (thus $\UU(n) \supseteq \UU(n+1)$ for all $n \in \bZ$), then $\bigcap_{n \in \bZ} \UU(n)$ is an aisle, and
\item if the sequence $(\UU(n))_{n \in \bZ}$ is increasing (thus $\UU(n) \subseteq \UU(n+1)$ for all $n \in \bZ$), then $\bigcup_{n \in \bZ} \UU(n)$ is an aisle.
\end{enumerate}
\end{corollary}

\begin{proof}
This follows from Proposition \ref{proposition:UnionOfAisles} together with Proposition \ref{proposition:RetractsCoproducts}.
\end{proof}

The following lemma shows there is a nice connection between narrow sequences and the restrictions in Definition \ref{definition:Restrictions}.

\begin{lemma}\label{lemma:CuttingUpRestrictions}
Let $(\NN(n))_{n \in \bZ}$ be a narrow sequence, and write $\UU = \theta((\NN(n))_{n \in \bZ})$.  For all $-\infty \leq k \leq l \leq +\infty$, we have that
$$H_n (\UU|_{k}^l) = \left\{ \begin{array}{ll} 0 & n < k \\ \NN(n) & k \leq n \leq l \\ \wide \NN(l) & l < n \end{array} \right.$$
\end{lemma}

\begin{proof}
This follows directly from the definitions and Corollary \ref{corollary:GrowingFastEnough}.
\end{proof}

\begin{lemma}\label{lemma:GluingPieces}
For $k \leq l < m$, we have $\UU|_k^l \ast \UU|_{l+1}^m = \UU |_k^m.$
\end{lemma}

\begin{proof}
By heredity, we know that every object in $\UU |_k^m$ is a direct sum of an object in $\UU|_k^l$ and an object in  $\UU|_{l+1}^m$.  The required property then follows easily.
\end{proof}

\begin{proposition}\label{proposition:UltimatePasting}
Let $\UU$ be a homology determined preailse.  If $\UU|_k^k$ are aisles, for all $k \in \bZ$, then $\UU$ is an aisle.
\end{proposition}

\begin{proof}
Being intersections of preaisles, all the $\UU|_k^m$ are preaisles.  So using induction, one can use Proposition \ref{proposition:Chen} and Lemma \ref{lemma:GluingPieces} to show that $\UU |_k^m$ is an aisle, for all $k \leq m$.  We define $\UU(n) = \UU |_{-n}^n$.  It then follows from Lemma \ref{lemma:CuttingUpRestrictions} that this sequence of aisles satisfies the conditions of Corollary \ref{corollary:UnionOfAisles}, and hence $\UU = \bigcup_n \UU(n)$ is an aisle.
\end{proof}

\begin{corollary}\label{corollary:Pasting}
Let $(\NN(n))_{n \in \bZ}$ be a narrow sequence such that
\begin{enumerate}
\item each $\NN(k) \to \AA$ has a right adjoint, and
\item each $\thick (\NN(k)) \to \Db \AA$ has a right adjoint.
\end{enumerate}
Then $\UU = \theta((\NN(n))_{n \in \bZ})$ is an aisle in $\Db \AA$.
\end{corollary}

\begin{proof}
Recall from Corollary \ref{corollary:NarrowIsTorsion} that $\NN(k)$ is a tilting torsion class in $\wide \NN(k)$, so that $\UU|_{k}^{k}$ is an aisle in $\thick \NN(k)$ (see Section \ref{subsection:Torsion}).  Since the embedding $\thick \NN(k) \to \Db \AA$ has a right adjoint, it follows that $\UU|_k^k$ is an aisle in $\Db \AA$.  The statement then follows from Proposition \ref{proposition:UltimatePasting}.
\end{proof}

\begin{proposition}\label{proposition:Subaisles}
Let $\UU$ be an aisle, then for all $k \leq l$ we have
\begin{enumerate}
\item $\UU |_k^\infty$ is an aisle,
\item $\UU |_{- \infty}^l$ is an aisle when $\thick H_l\UU \to \Db \AA$ has a right adjoint,
\item $\UU |_k^l$ is an aisle when $\thick H_l\UU \to \Db \AA$ has a right adjoint.
\end{enumerate}
\end{proposition}

\begin{proof}
Since $\UU |_k^{+\infty} = \UU \cap \DD^{\leq -k}$ is the intersection of two preaisles, it is a preaisle itself.  The right adjoint to the embedding $\UU |_k^{+\infty} \to \Db \AA$ is given by the composition $\Db \AA \to \UU \to \DD^{\leq -k}$, thus mapping an object $X \in D^* \AA$ to $(X_\UU)_{\DD^{\leq -k}}$.  Indeed, by Proposition \ref{proposition:IntersectionAdjoint} it suffices to check that $(X_\UU)_{\DD^{\leq -k}} \in \UU$.  For any $X \in \Db \AA$, we know by heredity that $(X_\UU)_{\DD^{\leq -k}}$ is a direct summand of $X_\UU$.  Since $\UU$ is an aisle, and thus closed under direct summands by Lemma \ref{lemma:AdjointRetract}, we see that $(X_\UU)_{\DD^{\leq k}} \in \UU$ as required.

We now turn to the second statement.  Thus we will show that the preaisle $\UU |_{-\infty}^{l}$ is an aisle.  To ease notation, let $\TT = \thick H_l\UU$.

Let $X \in \Db \AA$; we will show that $(X_\TT)_\UU \in \TT$ so that the required property follows from Proposition \ref{proposition:IntersectionAdjoint}.

Since $(-)_\UU$ is a right adjoint and hence commutes with products, it suffices to check that $(H_k (X_\TT)[k])_\UU \in \TT$, for all $k \in \bZ$.  We will consider two cases.  The first case is where $k > l$ so that $H_k (X_\TT)[k] \in \UU$ and thus $H_k (X_\TT)[k] \cong (H_k (X_\TT)[k])_\UU$.  The required property follows easily.

The second case is where $k \leq l$.  By Lemma \ref{lemma:HereditaryAdjoints}, we know that $(H_k (X_\TT)[k])_\UU$ can only have nonzero homologies in degrees $k, k-1$ so that $(H_k (X_\TT)[k])_\UU \in \TT$.  Again, the required property holds.

The last statement follows from the other two, since $\UU|_{k}^l = (\UU|_k^\infty)|_{-\infty}^l$.
\end{proof}

\begin{lemma}\label{lemma:AdjointInAisle}
Let $\UU$ be an aisle, and write $\TT = \thick H_k\UU$.  Assume that $\TT \to \Db \AA$ has a left adjoint.  Then ${}^\perp \TT \to \Db \AA$ has a right adjoint, and for any $X \in \UU$ we have that $X^\TT \in \UU$ and $X_{{}^\perp \TT} \in \UU$.
\end{lemma}

\begin{proof}
That ${}^\perp \TT \to \Db \AA$ has a right adjoint follows from Proposition \ref{proposition:Perpendicular}.  For the second statement, let $X \in \UU$.  Without loss of generality, we may assume that $X$ is a stalk complex concentrated in degree $n \in \bZ$.

We will first proof that $X^\TT \in \UU$.  If $n \leq k$, then $X \in \TT$ so that $X^\TT \cong X$ and thus $X^\TT \in \UU$.  For the case where $n > k$, note that $(X^\TT)_i$ is only nonzero for $i \geq n > k$ (see for example the dual of Lemma \ref{lemma:HereditaryAdjoints}), so that $X^\TT \in \UU$ since $\TT_i \subseteq \UU_i$ by Corollary \ref{corollary:GrowingFastEnough}. 

To prove that $X_{{}^\perp \TT} \in \UU$ we will start from the triangle
$$X^{\TT}[-1] \to X_{{}^\perp \TT} \to X \to X^\TT$$
given by Proposition \ref{proposition:Perpendicular}.  Since $\UU_n \subseteq \TT_n$ for $n \leq k$, the case where the homology of $X$ is concentrated in degree $n \leq k$ yields $X_{{}^\perp \TT} = 0$ so that $X_{{}^\perp \TT} \in \UU$ as required.

The second case we will consider is when the homology of $X$ is concentrated in degree $n > k + 1$.  In this case the dual of Lemma \ref{lemma:HereditaryAdjoints} shows that the homologies of $X^\TT$ can only be nonzero in degrees at least $k+2$.  Since $\UU_l \supseteq \TT_l$ for $l > k$ (see Corollary \ref{corollary:GrowingFastEnough}), we see that $X^{\TT}[-1] \in \UU$ and thus also $X_{{}^\perp \TT} \in \UU$ (by the above triangle).

The third and last case we consider is where the homology of $X$ is concentrated in degree $k+1$.  We take homologies of the above triangle to obtain
$$0 \to (X^\TT)_{k+2} \to (X_{{}^\perp \TT})_{k+1} \to X_{k+1} \to (X^\TT)_{k+1} \to (X_{{}^\perp \TT})_{k} \to 0.$$
We start by showing that $(X_{{}^\perp \TT})_{k} = 0$.  Since $\UU_k$ is a tilting nullity class in $\TT_{k}$ (see Corollary \ref{corollary:NarrowIsPretorsion}) we know that there is a monomorphism $(X^\TT)_{k+1} \to Z$ where $Z \in \UU_k$ (note that $X^\TT \in \TT$ and since $\TT$ is a thick subcategory, we have $(X^\TT)_{k+1} \in \TT_{k+1} = \TT_k$).  We get the following commutative diagram
$$\xymatrix{X_{k+1} \ar[r] \ar@{=}[d]& (X^\TT)_{k+1} \ar[d]\ar[r] & (X_{{}^\perp \TT})_{k} \ar[r] & 0 \\ 
X_{k+1} \ar[r] & Z \ar[r] & Z' \ar[r] & 0}$$
where the rows are exact and the downward arrow is a monomorphism.  It then follows from the Four Lemma that the induced arrow $(X_{{}^\perp \TT})_{k} \longrightarrow Z'$ is also a monomorphism.  However, the exactness of the last row shows that $Z' \in \UU_k \subseteq \TT_k$.  Indeed, this follows from Definition \ref{definition:NarrowSequence} where $A = X_{k+1}$, $B = Z$, and $C = Z'$.
This then implies that the morphism $(X_{{}^\perp \TT})_{k} \longrightarrow Z'$ is zero and hence that $(X_{{}^\perp \TT})_{k} = 0$ as required.

We are left with showing that $(X_{{}^\perp \TT})_{k+1} \in \UU_{k+1}$.  Therefore, note that $(X^{\TT})_{k+2} \in \TT_{k+2} = \TT_k \subseteq \UU_{k+1}$ (by Corollary \ref{corollary:GrowingFastEnough}) and $(X^{\TT})_{k+1} \in \TT_{k+1} = \TT_{k}$ and hence by Corollary \ref{corollary:NarrowIsPretorsion} there is a monomorphism $(X^{\TT})_{k+1} \to Z$ where $Z \in \UU_k$.  We then have an exact sequence
$$0 \to (X^{\TT})_{k+2} \to (X_{{}^\perp \TT})_{k+1} \to X_{k+1} \to Z$$
and it follows from the definition of a narrow sequence (Definition \ref{definition:NarrowSequence}) that $(X_{{}^\perp \TT})_{k+1} \in \UU_{k+1}$ as required.
\end{proof}

\begin{proposition}\label{proposition:LeftAdjointAisles}
Let $\UU$ be an aisle and assume that $\thick H_k \UU \to \Db \AA$ has both a left and a right adjoint.  Then
\begin{enumerate}
\item $\UU \cap \thick H_k \UU$ is an aisle in $\Db \AA$, and
\item $\UU \cap {}^\perp \thick H_k \UU$ is an aisle in $\Db \AA$.
\end{enumerate}
Moreover, we have $\UU = (\UU \cap {}^\perp \thick H_k \UU) \ast (\UU \cap \thick H_k \UU)$
\end{proposition}

\begin{proof}
The first statement follows from Proposition \ref{proposition:Subaisles}; the second statement follows from Proposition \ref{proposition:Subaisles} together with Lemma \ref{lemma:AdjointInAisle}.  For the last statement, the inclusion $\UU \supseteq (\UU \cap {}^\perp \thick H_k \UU) \ast (\UU \cap \thick H_k \UU)$ is clear.  For the other inclusion, let $X \in \UU$.  Write $\TT$ for $\thick H_k \UU$.  There is a triangle
$$X^\TT[-1] \to X_{{}^\perp \TT} \to X \to X^\TT,$$
given by Proposition \ref{proposition:Perpendicular}, where it follows from Lemma \ref{lemma:AdjointInAisle} that $X_{{}^\perp \TT} \in \UU \cap {}^\perp \TT$ and that $X^{\TT} \in \UU \cap \TT$.  This proves the required property.
\end{proof}

\begin{lemma}\label{lemma:BigGluing}
Let $\UU$ be an aisle and write $\TT = \thick H_k \UU$.  Assume that $\TT \to \Db \AA$ has both a left and a right adjoint.  Let $\VV = (\UU \cap {}^\perp \TT) \ast \TT$, then $\VV_n = \UU_n$ for $n > k$.
\end{lemma}

\begin{proof}
It follows from Proposition \ref{proposition:LeftAdjointAisles} that $\UU = (\UU \cap {}^\perp \TT) \ast (\UU \cap \TT)$ so that $\UU \subseteq \VV$.  In particular, we have $\UU_n \subseteq \VV_n$ for all $n \in \bZ$.  To show the other inclusion for $n \geq k$, note that $\VV$ is an aisle by Proposition \ref{proposition:Chen}.  In particular it is closed under retracts, so that it suffices to show that $X \in \VV$ implies $X \in \UU$ whenever $X$ is a stalk complex supported in degree $n > k$.

So let $X$ be as above.  Due to the definition of $\VV$ there is a triangle $U \to X \to T \to U[1]$ where $U \in \UU \cap {}^\perp \TT \subseteq \UU$ and $T \in \TT$.  Taking homologies, we get
$$0 \to T_{n+1} \to U_n \to X_n \to T_{n} \to U_{n-1} \to 0$$
so we can assume that $T_i \not= 0$ only for $i = n, n+1$.  Since $i \geq n > k$, we have $T_i \in \TT_i = \TT_k \subseteq \UU_i$ by Corollary \ref{corollary:GrowingFastEnough} so that $T \in \UU$.  Since $\UU$ is closed under extensions, we know that $X \in \UU$.  This finishes the proof.
\end{proof}
\section{Coreflective narrow sequences and aisles}\label{section:CoreflectiveNarrowSequence}

In this section, we will look at the relation between aisles and narrow sequences.  It has been shown in Proposition \ref{proposition:RestrictingAdjoints} that, if $\UU$ is an aisle, then every subcategory $\UU_n \subseteq \AA$ is coreflective, i.e. the embedding $\UU_n \longrightarrow \AA$ has a right adjoint.  A narrow sequence where all subcategories are coreflective, will be called a \emph{coreflective narrow sequence}.

Conversely, a coreflective narrow sequence does not always yield an aisle.  We will show in Theorem \ref{theorem:EnoughInjectives} that this is the case, however, when $\AA$ has enough injectives.

The following Proposition is an extension of Theorem \ref{theorem:NarrowPreaisle} to the case of coreflective narrow sequences and aisles.

\begin{proposition}\label{proposition:SequencesAisles}
Let $\AA$ be a hereditary abelian category.  The function $\theta$ from Definition \ref{definition:Theta} induces an injection
$$\begin{array}{lcr}
\left\{
\begin{array}{l}
\mbox{Coreflective narrow sequences $\{\NN(k)\}_k$ such that} \\
\mbox{$\Db (\wide \NN(k)) \to \Db \AA$ has a right adjoint}
\end{array}
\right\}&
\longrightarrow&
\left\{ \mbox{$t$-structures on $\Db \AA$} \right\}.
\end{array}$$
The function $\mu$ from Definition \ref{definition:Mu} induces an injection
$$\begin{array}{lcr}
\left\{ \mbox{$t$-structures on $\Db \AA$} \right\}&
\longrightarrow&
\left\{ \mbox{Coreflective narrow sequences in $\AA$} \right\}.
\end{array}$$
\end{proposition}

\begin{proof}
It has been shown in Theorem \ref{theorem:NarrowPreaisle} that $\theta((\NN(k))_k)$ is a preaisle.  It follows from Corollary \ref{corollary:Pasting} that it is an aisle.

The last statement follows from Theorem \ref{theorem:NarrowPreaisle} combined with Proposition \ref{proposition:RestrictingAdjoints}.
\end{proof}

\begin{remark}\label{remark:NotEnoughInjectives}
The restricted functions $\theta, \mu$ do not yields bijections in general.  This is the case, for example, when $\AA$ is $\mod \bZ$, the category of finitely generated abelian groups.  We refer to Example \ref{example:NotEnoughInjectives}.
\end{remark}

An important special case will be when $\AA$ has enough injectives, or dually when $\AA$ has enough projectives.  We start with the following observation.

\begin{proposition}\label{proposition:EnoughInjectives}
Let $\AA$ be a hereditary category with enough injectives.  Let $\NN$ be a coreflective narrow subcategory of $\AA$.  Then $\NN$ has enough injectives and the right adjoint to the embedding $i:\NN \to \AA$ maps injectives in $\AA$ to $\NN$-injectives.  The wide closure $\wide \NN$ of $\NN$ has enough injectives and the embedding $\Db (\wide \NN) \to \Db \AA$ has a right adjoint.
\end{proposition}

\begin{proof}
Let $0 \to I_\NN \to M \to N \to 0$ be a short exact sequence in $\NN$.  Since the functor $\Hom(-,I_\NN) \cong \Hom(i-,I)$ maps this sequence to an exact sequence, we see that $\Hom(M,I_\NN) \to \Hom(I_\NN, I_\NN)$ is surjective and thus that the original short exact sequence splits.  This shows that $I_\NN$ is an $\NN$-split injective and \ref{proposition:SplitExt} shows it is $\NN$-injective.

To show that $\NN$ has enough injectives, let $N \in \NN$.  In $\AA$, there is a monomorphism $N \to I$ where $I$ is injective.  Since right adjoint functors preserve monomorphisms, there is a monomorphism $N_\NN \to I_\NN$.  Because $i$ is fully faithful we know that $N \cong N_\NN$, and we have already established that $I_\NN$ is injective in $\NN$.  We conclude that $\NN$ has enough injectives.

Let $I$ be an injective in $\NN$; we wish to show that $I$ is injective in $\WW$.  Consider a monomorphism $I \to W$ where $W \in \WW$.  By Proposition \ref{proposition:GeneratedInOneStep} there is a monomorphism $W \to N$ where $N \in \NN$.  The composition $I \to W \to N$ is a split monomorphism, and hence so is $I \to W$.  We conclude that injectives in $\NN$ are injectives in $\WW$.  It then follows from Proposition \ref{proposition:GeneratedInOneStep} that $\WW$ has enough injectives.

Furthermore, since $\AA$ has enough injectives, we may derive the right adjoint of $\NN \to \AA$ to a right adjoint of $\Db \WW \to \Db \AA$.
\end{proof}

We can now use Proposition \ref{proposition:EnoughInjectives} to describe a case where the maps in Proposition \ref{proposition:SequencesAisles} are bijections, obtaining a straightforward relation between a structure in the abelian category and a structure in the derived category.

\begin{theorem}\label{theorem:EnoughInjectives}
Let $\AA$ be a hereditary category with enough injectives.  Then $\theta, \mu$ induce bijections
$$\begin{array}{lcr}
\left\{ \mbox{$t$-structures on $\Db \AA$} \right\}&
\stackrel{\sim}{\longleftrightarrow}&
\left\{ \mbox{Coreflective narrow sequences in $\AA$} \right\}.
\end{array}$$
\end{theorem}

\begin{proof}
By combining Propositions \ref{proposition:SequencesAisles} and \ref{proposition:EnoughInjectives}, we see that the function $\theta$ induces the required bijection.  Theorem \ref{theorem:NarrowPreaisle} shows that $\mu$ is inverse to $\theta$.
\end{proof}

In some cases --as in the category of finitely generated modules over a Dedekind domain which we will discuss in \S\ref{section:Dedekind} below-- one does not have enough injectives, but enough projectives.  In these cases, one can apply the dual of Theorem \ref{theorem:EnoughInjectives}.  The dual concept of a narrow sequence in a hereditary category $\AA$ will be called a \emph{co-narrow sequence}, i.e. $(\CC(k))_{k \in \bZ}$ is a co-narrow sequence in $\AA$ if and only if $(\CC(-k)^\circ)_{k \in \bZ}$ is a narrow sequence in $\AA^\circ$.  A \emph{reflective co-narrow sequence} is a co-narrow sequence where each embedding $\CC(k) \to \AA$ has a left adjoint.  We then have the following.

\begin{corollary}\label{corollary:EnoughProjectives}
Let $\AA$ be a hereditary category with enough projectives.  Then $\theta, \mu$ induce bijections
$$\begin{array}{lcr}
\left\{ \mbox{$t$-structures on $\Db \AA$} \right\}&
\stackrel{\sim}{\longleftrightarrow}&
\left\{ \mbox{Reflective co-narrow sequences in $\AA$} \right\}.
\end{array}$$
\end{corollary}

\section{A further reduction}\label{section:LeftAdjoint}

In \S\ref{section:CoreflectiveNarrowSequence}, we showed that in some cases one can classify aisles by classifying coreflective narrow sequences.  In general however, the latter might not be easier to classify than the former.  In this section, we will discuss a special case where one can reduce the coreflective narrow sequence to better understood concepts.

In this section, let $\AA$ be an abelian hereditary category satisfying the following two properties.  For each coreflective narrow subcategory $\NN \subseteq \AA$ we have that
\begin{enumerate}
\item the lifting $\Db (\wide \NN) \to \Db \AA$ has a right adjoint, and
\item the lifting $\Db (\wide \NN) \to \Db \AA$ has a left adjoint.
\end{enumerate}
By Proposition \ref{proposition:EnoughInjectives}, the first property holds when $\AA$ has enough injectives.  It is well-known (see for example \cite{Krause11}) that the second property holds, for example, when $\AA \cong \mod A$ where $A$ is a finite dimensional hereditary $\bK$-algebra (where $\bK$ is a field).

By the first property, we know that the maps $\theta$ and $\mu$ give bijections between the class of $t$-structures in $\Db \AA$ and the class of coreflective narrow sequences in $\AA$ (see \ref{proposition:SequencesAisles}).

In order to describe the $t$-structures on $\Db \AA$, we will use the following definition.

\begin{definition}
Let $\AA$ be a hereditary category.  Let $\PW(\AA)$ be the poset of all coreflective wide subcategories.  A poset morphism $\bZ \to \PW(\AA)$ is called a \emph{$t$-sequence}.  The set of all $t$-sequences is denoted by $\PW(\AA)^\bZ$.
A \emph{refined $t$-sequence} consists of the following data: a $t$-sequence $f: \bZ \to \PW(\AA)$, and a function $t_f$ mapping $n \in \bZ$ to a tilting torsion class in $f(n) \cap {}^\perp f(n-1)$.  We denote the set of refined $t$-sequences by $\Delta(\AA)$.
\end{definition}

\begin{remark}
When $\AA$ is the category of finitely generated modules over a commutative noetherian ring $A$, then it has been shown in \cite{Stanley10} that the aisles are determined by a poset morphism, called a perversity function in \cite{ArinkinBezrukavnikov10}, from $\bZ$ to the poset of (not necessarily coreflective) wide subcategories in $\AA$.  In this situation, the only tilting torsion theory is the wide subcategory itself (\cite{StanleyWang11}).
\end{remark}

\subsection{\texorpdfstring{The map $\Xi$ from aisles to refined $t$-sequences}{The map from aisles to refined t-sequences}}

We start by defing a function $\Xi$, mapping an aisle in $\Db \AA$ to a refined $t$-structure in $\Delta(\AA)$.ing a function $\Xi$, mapping an aisle in $\Db \AA$ to a refined $t$-structure in $\Delta(\AA)$.

\begin{definition}\label{definition:Xi}
Let $\AA$ be a hereditary category as described in the beginning of this section.  We define a function
$$\Xi: \begin{array}{lcr}
\left\{ \mbox{Aisles in $\Db \AA$} \right\}&
\longrightarrow&
\Delta(\AA)
\end{array}$$
by $\Xi(\UU) = (f,t_f)$ where $f(n) = \wide (H_n \UU)$ and $t_f(n) = \UU_n \cap {}^\perp f(n-1)$.
\end{definition}

We need to show that this map is well-defined, i.e. that $\Xi(\UU) \in \Delta(\AA)$ for an aisle $\UU \subseteq \Db \AA$.  This follows from the following proposition.

\begin{proposition}
For each $n \in \bZ$, we have that $f(n)$ is a coreflective wide subcategory of $\AA$ and $t_f(n)$ is a tilting torsion class in $f(n) \cap {}^\perp f(n-1)$.
\end{proposition}

\begin{proof}
That $f(n)$ is a coreflective wide subcategory follows directly from the first requirement on $\AA$ in this section, together with Proposition \ref{proposition:RestrictingAdjoints}.

Define $\VV = \UU \cap {}^\perp \thick_{n-1} \UU$.  Recall that it follows from Proposition \ref{proposition:RestrictingAdjoints} that $H_{n-1} \UU$ is a coreflective subcategory of $\AA$, so that the conditions on $\AA$ imposed in this section yield that $\thick H_{n-1} \UU \to \Db \AA$ has both a left and a right adjoint.  It then follows from Proposition \ref{proposition:LeftAdjointAisles} that $\VV$ is an aisle in $\Db \AA$, thus also in ${}^\perp \thick H_{n-1} \UU$.   It now follows from Corollary \ref{corollary:NarrowIsTorsion} that $\VV_n$ is a tilting torsion class in $({}^\perp \thick H_{n-1} \UU)_n = f(n) \cap {}^\perp f(n-1)$.
\end{proof}

\subsection{\texorpdfstring{The map $\Psi$ from refined $t$-sequences to aisles}{The map from refined t-sequences to aisles}}\label{subsection:Psi}

We will now define a map $\Psi$ from refined $t$-sequences to aisles which we will later prove to be inverse to the above map $\Xi$.  Since the map $\Xi$ ``cuts up'' an aisle into a refined $t$-sequence, the map $\Psi$ must ``glue'' the pieces back together.  This gluing is easiest to understand in the derived category.

Since we may have infinitely many categories to glue (the result must contain the categories $f(n-1)[n] \subset \Db \AA$ and $t_f(n)[n] \subset \AA$), we will describe $\Psi(f,t_f)$ using the ``approximating aisles'' $\VV(n,m)$, each of which is obtained by finitely many operations.  The aisles $\VV(n,m)$ approximate $\Psi(f,t_f)$ in the following sense: it is shown in Corollary \ref{corollary:GoodApproximation} below that $\Psi(f,t_f)_k = \VV(n,m)_k$ for all $n \leq k \leq m$.

For the actual description, fix a refined $t$-sequence $(f,t_f)$; we will start the construction of $\VV(n,m)$.  First, recall that $f(n)$ is a wide coreflective subcategory of $\AA$ and that the embedding lifts to a fully faithful functor $\Db f(n) \to \Db \AA$.  We will denote the essential image by $\TT(n)$; thus $\TT(n)$ is the thick subcategory of $\Db \AA$ such that all the homologies lie in $f(n)$.  By our assumptions $\TT(n)$ is also a (co)reflective thick subcategory of $\Db \AA$.

For each $n \in \bZ$, define the preaisle $\VV^{(n)} \subseteq \Db \AA$ by the narrow sequence
$$\NN^{(n)}(k) = \left\{\begin{array}{ll} 0 & k < n, \\ t_f(n) & k=n, \\ f(n) \cap {}^\perp f(n-1) & n < k. \end{array}\right.$$
thus $\VV^{(n)} = \theta((\NN^{(n)}(k))_k)$, where $\theta$ is as defined in Definition \ref{definition:Theta}.

We then define $\VV(n,m)$, where $n \leq m$, inductively by
$$\mbox{$\VV(n,n) = \VV^{(n)} \ast \TT(n-1)$ and $\VV(n,m+1) = \VV^{(m+1)} \ast \VV(n,m)$,}$$
thus $\VV(n,m) = \VV^{(m)} \ast \VV^{(m-1)} \ast \cdots \ast \VV^{(n)} \ast \TT(n-1)$.  Finally, let
$$\mbox{$\VV(n,\infty) = \cup_{m \geq n} \VV(n,m)$ and $\VV = \cap_{n \in \bN} \VV(n,\infty)$.}$$

Our goal is to show that $\VV$ is an aisle so that the following definition makes sense.  This will follow from Proposition \ref{proposition:VVisAisle} below.

\begin{definition}\label{definition:Psi}
Let $(f,t_f) \in \Delta(\AA)$.  We define a function 
\begin{eqnarray*}
\Psi: \Delta(\AA) &\longrightarrow& \{\mbox{Aisles in $\Db \AA$}\} \\
(f,t_f) &\mapsto& \VV
\end{eqnarray*}
where $\VV$ is as constructed above.
\end{definition}

\begin{lemma}\label{lemma:AreAisles}
Using the notations above, we have that
\begin{enumerate}
\item $\VV(n,m+1), \VV^{m+1}$ are aisles,
\item $\VV, \VV(n, \infty)$ are preaisle.
\end{enumerate}
\end{lemma}

\begin{proof}
Note that since $\VV^{(m+1)} \subseteq {}^\perp \TT(m) \subseteq {}^\perp \VV(n,m)$ we see that $\VV(n,m+1)$ is a preaisle (see Lemma \ref{lemma:PreaislesStar}).  Since each of the embeddings
$$\VV^{(m+1)} \to \TT(m+1) \cap {}^\perp \TT(m) \to \TT(m+1) \to \Db \AA$$
has a right adjoint, so does the composition and hence $\VV^{(m+1)}$ is also an aisle in $\Db \AA$.  Because the $T(i)$ are also aisles, it follows from Proposition \ref{proposition:Chen} that $\VV(n,m+1)$ is an aisle in $\Db \AA$.

Note that $\VV(n,m) \subseteq \VV(n,m+1)$ so that $\VV(n,\infty)$ is a preaisle, and consequently so is $\VV$.
\end{proof}

\begin{remark}
Since $\TT(n-1), \VV^{(n)} \subseteq \TT(n)$, we have $\VV(n,m) \subseteq \VV(n+1,m)$ where $n < m$.  Thus when there are $n, m \in \bN$ such that $f(k) = f(n)$ for $k < n$ and $f(k) = f(m)$ for $nm \leq k$, then we have that $\VV = \VV(n,m)$.  In particular, $\VV$ is an aisle.
\end{remark}

\begin{proposition}\label{proposition:NoChange}
Let $n \leq k \leq m$.  With notations as above we have that
\begin{enumerate}
\item $\VV(n,m)_k = \VV(k,k)_k$, and
\item $\VV(n,m)_k$ is a tilting torsion class in $f(k)$.
\end{enumerate}
\end{proposition}

\begin{proof}
We start by showing that $\VV(n,m)_k = \VV(n,k)_k$.  It is clear that $\VV(n,m)_k \supseteq \VV(n,k)_k$ so that we only need to show the other inclusion.  Let $X \in \VV(n,m)_k$, thus there is a triangle
$$A \to X[k] \to B \to A[1]$$
where $A \in \VV^{(m)} \ast \VV^{(m-1)} \ast \cdots \ast \VV^{(k+1)}$ and $B \in \VV^{(k)} \ast \VV^{(k-1)} \ast \cdots \ast \VV^{(n)} \ast \TT(n-1) = \VV(n,k)$.  It is easy to show that $A_l \cong 0$ when $l \leq k$, so that by taking homologies, we see that $X \cong B_k$, and thus $X \in \VV(n,k)_k$.  This shows the other inclusion.

We will now show that $\VV(n,m)_k$ is a tilting torsion class in $f(k)$ where $n \leq k \leq m$.  By the first part of the proof, we may reduce to showing that $\VV(n,m)_m$ is a torsion class in $f(m)$.%  We will work by induction on $m \geq n$.  Thus we start with $m=n$.

We start with the case $m=n$.  Recall that that $\VV(n,n) = \VV^{(n)} \ast \TT(n-1)$ and that $\VV^{(n)}, \TT(n-1) \subseteq \TT(n)$ so that $\VV(n,n) \subseteq \TT(n)$.  In particular, $\VV(n,n)_n \subseteq f(n)$ and hence $\wide \VV(n,n)_n \subseteq f(n)$.  We want to show that $\wide \VV(n,n)_n = f(n)$.

We see that $\wide \VV(n,n)_n$ contains both $f(n-1) = \TT(n-1)_n$ and $\wide \VV^{(n)} = \wide t_f(n) = f(n) \cap {}^\perp f(n-1)$.  Let $X \in f(n)$.  In the derived category $\Db \AA$ we have a triangle given by Proposition \ref{proposition:Perpendicular}
$$A \to X[0] \to B \to A[1]$$
where $B \in \TT(n-1)$ and $A \in {}^\perp \TT(n-1)$.  Taking homologies, gives the exact sequence
$$0 \to B_1 \to A_0 \to X \to B_0 \to A_{-1} \to 0$$
where now $A_0,A_1 \in f(n) \cap {}^\perp f(n-1)$ and $B_{-1},B_0 \in f(n)$.  This shows that $X$ lies in the wide subcategory containing both $f(n)$ and $f(n) \cap {}^\perp f(n-1)$.  We conclude that indeed $\wide \VV(n,n)_n = f(n)$.

We have already established in Lemma \ref{lemma:AreAisles} that $\VV(n,n)$ is an aisle so that by Proposition \ref{proposition:SequencesAisles} we know that $\mu(\VV(n,n))$ is a coreflective narrow sequence and thus Corollary \ref{corollary:NarrowIsTorsion} yields that $\VV(n,n)_n=\NN(n)$ is a tilting torsion theory in $f(n)$.

To show that that $\VV(n,m)_m$ is a torsion class in $f(m)$ for $n \leq m-1$, we recall that we have already shown in the first part of the proof that $\VV(n,m)_{m-1} = \VV(n,m-1)_{m-1}$, so that one can easily use induction on $m-n$ to prove the general case.  This proves the second statement of the proposition.

For the first statement, we have that $\VV(n,m)_k = \VV(n,k)_k$ as in the first part of the proof.  Note that we have $\VV(n,k)_{k-1}$ is a tilting torsion class in $f(k-1)$ so that $\thick H_{k-1} \VV(n,k)$ coincides with $\TT(k-1)$.  It now follows from Lemma \ref{lemma:BigGluing} that $\VV(n,k)_k = \VV(k,k)_k$ as required.
\end{proof}

\begin{corollary}\label{corollary:GoodApproximation}
For each $n \leq k \leq m$ we have that $\VV_k = \VV(n,\infty)_k = \VV(n,m)_k$.
\end{corollary}

\begin{proposition}\label{proposition:VVisAisle}
$\VV$ is an aisle in $\Db \AA$.
\end{proposition}

\begin{proof}
It follows from Corollary \ref{corollary:GoodApproximation} that we can apply Corollary \ref{corollary:UnionOfAisles} to see that $\VV(n,\infty)$ and subsequently $\VV$ are aisles in $\Db \AA$.
\end{proof}

\subsection{\texorpdfstring{Aisles are classified by refined $t$-sequences}{Aisles are classified by refined t-sequences}}  In this subsection, we will show that aisles in $\Db \AA$ are classified by refined $t$-sequences (see Theorem \ref{theorem:Reduced}).

\begin{lemma}\label{lemma:PsiXi}
For each aisle $\UU$ in $\Db \AA$, we have $\Psi \circ \Xi (\UU) = \UU$.  
\end{lemma}

\begin{proof}
We write $\Xi (\UU)= (f,t_f)$ and $\VV = \Psi(f,t_f)$.  We want to show that $\VV_k = \UU_k$ for each $k \in \bZ$.  Using the notations of Section \ref{subsection:Psi}, this is equivalent to showing that $\VV(k,\infty)_k = \UU_k$ by Proposition \ref{proposition:NoChange}.  This then follows from Lemma \ref{lemma:BigGluing}.
\end{proof}

\begin{lemma}\label{lemma:XiPsi}
For each refined $t$-structure $(f,t_f) \in \Delta(\AA)$ we have $\Xi \circ \Psi (f,t_f) = (f,t_f)$.
\end{lemma}

\begin{proof}
Write $\VV = \Psi(f,t_f)$.  It follows from Corollary \ref{corollary:GoodApproximation} that $\wide \VV_n = \wide \VV(n,n)_n$ and then by Proposition \ref{proposition:NoChange} that $\wide \VV_n = f(n)$.

Next we show that $t_f(k) = \VV_k \cap {}^\perp f(k-1)$.  Here again, we use Proposition \ref{proposition:NoChange} to see that
\begin{align*}
\VV_k \cap {}^\perp f(k-1) &= \VV(k,k)_k \cap {}^\perp f(k-1) \\
&= (\VV^{(k)} \ast \TT(k-1))_k \cap {}^\perp f(k-1) \\
&= (\VV^{(k)})_k \cap {}^\perp f(k-1) \\
&= t_f(k) \cap {}^\perp f(k-1) = t_f(k),
\end{align*}
which finishes the proof.
\end{proof}

Lemmas \ref{lemma:PsiXi} and \ref{lemma:XiPsi} prove the following result.

\begin{theorem}\label{theorem:Reduced}
Let $\AA$ be a hereditary category such that every for coreflective wide subcategory $\WW$ of $\AA$, the canonical functors $\Db \WW \to \Db \AA$ has a left and a right adjoint.  Then the functions $\Xi$ and $\Psi$ from Definitions \ref{definition:Xi} and \ref{definition:Psi} respectively yield bijections
$$\begin{array}{lcr}
\{\mbox{$t$-structures on $\Db \AA$}\}&
\stackrel{\sim}{\longleftrightarrow}&
\Delta(\AA).
\end{array}$$
\end{theorem}

\begin{remark}
The main application of Theorem \ref{theorem:Reduced} above is when $\AA$ is the category of finite dimensional representations of a finite dimensional hereditary algebra.  We then obtain a complete description of the $t$-structures when $A$ is either representation finite of tame.  For the representation finite case, we refer the reader also to \cite{KellerVossieck88, LiuVitoria11, Parthasarathy88}.
\end{remark}
\section{The unbounded case}

Most results in this article are valid if one replaces $\Db \AA$ by $D^* \AA$, where $*=+,-,\emptyset$, and the proofs carry over in a straightforward manner.  We will instead use the following proposition to translate the obtained results for $\Db \AA$ to $D^* \AA$.

\begin{proposition}
There are bijections
$$\{\mbox{Homology-determined preaisles on $\Db \AA$}\} \longleftrightarrow \{\mbox{Homology-determined preaisles on $D^* \AA$}\}$$
and
$$\{\mbox{Aisles on $\Db \AA$}\} \longleftrightarrow \{\mbox{Aisles on $D^* \AA$}\}.$$
\end{proposition}

\begin{proof}
Both bijections are given by relating $\UU \subseteq \Db \AA$ and $\VV \subseteq D^* \AA$ when $\UU_n = \VV_n \subseteq \AA$ for all $n \in \bZ$.  It is clear that this is a bijection between homology-determined preaisles; this proves the first statement.

For the second statement, we need to show that the above correspondence maps aisles in $\Db \AA$ to aisles in $D^* \AA$ and vice versa.  Thus let $\UU$ be an aisle in $\Db \AA$, and let $\VV$ be the corresponding homology-determined preaisle in $D^* \AA$.  For $V \in \VV$ we have natural isomorphisms
\begin{eqnarray*}
\Hom_{D^* \AA}(V,X) &\cong& \Hom_{D^* \AA}(\coprod_n V_{n}[n], \prod_m X_m[m]) \\
&\cong& \prod_{n,m} \Hom_{D^* \AA}(V_n[n],X_m[m]) \\
&\cong& \prod_{n,m} \Hom_{\Db \AA}(V_n[n],X_m[m]) \\
&\cong& \prod_{n,m} \Hom_{\Db \AA}(V_n[n],(X_m[m])_\UU) \\
&\cong& \prod_{n,m} \Hom_{D^* \AA}(V_n[n],(X_m[m])_\UU) \\
&\cong& \Hom_{D^* \AA}(\coprod_n V_{n}[n], \prod_m (X_m[m])_\UU) \\
&\cong& \Hom_{D^* \AA}(V, \prod_m (X_m[m])_\UU)
\end{eqnarray*}
so that, since $\prod_m (X_m[m])_\UU \in D^* \AA$, we may infer that $\VV \to D^* \AA$ has a right adjoint.  The other direction is similar.
\end{proof} 

\begin{corollary}
Theorems \ref{theorem:NarrowPreaisle}, \ref{theorem:EnoughInjectives}, and \ref{theorem:Reduced} holds verbatim after changing the bounded derived categories to the left/right bounded or unbounded derived categories.
\end{corollary}

We will end this section with an example concerning homology-determined preaisles in the unbounded derived category, which illustrates that homology-determined preaisles in $\Db \AA$ is typically not a homology-determined preaisle in $D^* \AA$ (although it will still be a preaisle).

\begin{example}\label{example:BadPreaisles2}
Consider the preaisle $\UU = \Db \mod k$ in $D \mod k$.  We have $H_k \UU \cong \mod k$ and $\theta ((H_k \UU)_{k \in \bZ}) \cong D \mod k \not\cong \Db \mod k$.  The preaisle $\UU$ is not homology-determined.  When considering $\UU$ as a preaisle in $\Db \mod k \subset D \mod k$, it is indeed a homology-determined preailse (as $\UU$ is even an aisle in $\Db \mod k$).
\end{example}
\section{Application: finitely presented modules of a Dedekind domain}\label{section:Dedekind}

Let $R$ be a commutative Dedekind domain, thus $R$ is an integrally closed noetherian (commutative) domain of Krull dimension at most one.  In Theorem \ref{theorem:Dedekind} below we will describe the $t$-structures on the category $\Db \mod R$, recovering a special case of \cite[Theorem 6.9]{TarrioLeovigildoJeremiasSaorin10}.

We start with a description of the category of finitely presented modules $\mod R$ (see for example \cite[Theorems 10.14 and 10.15]{Jacobson89} and \cite[Proposition 7.18]{McConnellRobson01}).

\begin{theorem}\label{theorem:DedekindClassification}
Every finitely presented module $X$ is isomorphic to $P \oplus T$ where $T$ is the torsion submodule of $X$ and $P$ is projective.
Furthermore, $T \cong \bigoplus_{i=1}^n R/ \pp_i^{n_i}$ for some nonzero prime ideal $\pp_i$ and $n_i > 1$.  These $(\pp_i,n_i)$ are uniquely determined up to permutations.
Every projective module $P$ is isomorphic to a direct sum of rank one projective modules.  If $I$ is a rank one projective module, then $I \oplus I \cong R \oplus I^2$.
\end{theorem}

\begin{remark}\label{remark:Projective}
Note that $\Ext(X,X) = 0$ implies that $X$ is a projective object.  Indeed, if $X$ were not projective then $X$ has a direct summand of the form $R/ \pp_i^{n_i}$ and since 
$$0 \to R/ \pp_i^{n_i} \to R/ \pp_i^{2n_i} \to R/ \pp_i^{n_i} \to 0$$
is a nonsplit exact sequence we see that $\Ext(R/ \pp_i^{n_i},R/ \pp_i^{n_i}) \not= 0$.
\end{remark}

\begin{theorem}\label{theorem:Dedekind}
Let $R$ be a Dedekind domain.  All aisles (and hence $t$-structures) in $\Db \mod R$ are given by a co-narrow sequence
$$\CC(k) = \left\{ \begin{array}{ll} \mod R & k < n \\ \CC & k=n \\ 0 & k > n \end{array}\right.$$
where $\CC \subseteq \mod R$ is a nonzero torsionfree class, and where $n \in \bZ \cup \{\pm \infty\}$.  Moreover, there is an order inversing bijection
\begin{align*}
\{\mbox{Torsionfree classes in $\mod R$}\} &\longrightarrow \{\mbox{Specialization closed subsets of $\Spec R$}\} \\
\CC & \mapsto \Supp(\CC^{\perp_0}).
\end{align*}
\end{theorem}

\begin{proof}
There is a bijection between the torsion classes of $\mod R$ and the torsionfree classes, given by mapping a torsion class $\TT \subseteq \mod R$ to $\TT^{\perp_0}$.  The given order preserving bijection between torsionfree classes in $\mod R$ and the specialization closed subsets of $\Spec R$, is given in \cite[Theorem 4]{StanleyWang11}.

Let $\CC$ be a torsionfree class in $\mod R$.  Since $\Hom(R,-) |_\CC$ is an exact functor, the object $R^\CC$ is $\CC$-projective (as in Section \ref{section:Injectives}) and thus $\Ext(R^\CC, R^\CC) = 0$.   As in Remark \ref{remark:Projective}, we see that $R^\CC$ is a projective object in $\mod R$.  Note that $R$ is a generator of $\mod R$ so that $R^\CC = 0$ implies that $\CC = 0$.  In the case where $R^\CC$ is nonzero, we know that $R^\CC$ is a nonzero projective object in $\mod R$ and since $\CC$ is closed under direct sums and summands, it follows from the classification in Theorem \ref{theorem:DedekindClassification} that $R \in \CC$.  The required result then follows from the dual of Proposition \ref{proposition:KernelInNext}.
\end{proof}

\begin{remark}
Since $R$ is a Gorenstein ring, the category $\Db \mod R$ has a dualizing complex.  Our Theorem \ref{theorem:Dedekind} then is a corollary of \cite[Theorem 6.9]{TarrioLeovigildoJeremiasSaorin10}.
\end{remark}

\begin{remark}
The cases $n = \pm \infty$ correspond to the degenerate $t$-structures on $\Db \mod R$.
\end{remark}

\begin{remark}
Since $\bZ$ is a Dedekind domain, one can use Theorem \ref{theorem:Dedekind} to classify all $t$-structures on the category $\mod \bZ$ of finitely presented abelian groups by sets of prime numbers.  It has been shown in \cite[Corollary 8.4]{Stanley10} that there is a proper class of $t$-structures in the unbounded derived category $D \Mod \bZ$ (and hence also in $\Db \Mod \bZ$) of all abelian groups.
\end{remark}

The following example shows that not every coreflective narrow sequence yields an aisle (see Remark \ref{remark:NotEnoughInjectives}).

\begin{example}\label{example:NotEnoughInjectives}
Let $(\NN(k))_{k \in \bZ}$ be a coreflective narrow sequence in $\mod \bZ$ given by
$$\NN(k) = \mbox{full subcategory of finite groups.}$$
Then $\NN(k) \to \mod \bZ$ has a right adjoint, but $(\NN(k))_{k \in \bZ}$ is not of the form described in Theorem \ref{theorem:Dedekind}.  Note that $\wide (\NN(k)) = \NN(k)$ but $\Db (\wide \NN(k))$ is not a coreflective subcategory in $\Db \mod \bZ$.
\end{example}

\section{Application: coherent sheaves on a smooth projective curve}\label{section:ApplicationCurves}

Let $\bX$ be a smooth projective curve over a field $\bK$.  We will discuss the $t$-structures on $\Db \coh \bX$.  Our main result is the same as \cite{GorodentsevKuleshovRudakov04} but our methods are different.

Recall (see for example \cite{Hartshorne77}) that $\coh \bX$ is an Ext-finite noetherian hereditary category.  We will recall some properties of $\coh \bX$ that we will use.

\begin{itemize}
\item Every object in $\coh \bX$ is the direct sum of a locally free sheaf and a torsion sheaf.  Furthermore, one has $\Hom(\GG, \FF) = 0$ when $\FF$ is a locally free sheaf and $\GG$ has finite length.  Using Serre duality, one has $\Hom(-,\GG \otimes \omega_X) \cong \Ext(\GG,-)^*$ and $\Hom(\GG,-) \cong \Ext(-,\GG \otimes \omega_X)^*$ where $(-)^*$ denotes the vector space dual and $\omega_X$ is the dualizing sheaf.
\item For every coherent sheaf $X \in \coh \bX$, one defines the rank $\rk X$ and the degree $\deg X$.  A sheaf is a torsion sheaf if and only if $\rk X = 0$; in this case $\deg X > 0$.
\item The Riemann-Roch formula can be written as
$$\chi(X,Y) = \begin{pmatrix} \rk X & \deg X \end{pmatrix}
\begin{pmatrix} 1-g & 1 \\-1 & 0 \end{pmatrix}
\begin{pmatrix} \rk Y \\ \deg Y \end{pmatrix}$$
where $\chi(X,Y) = \dim \Hom(X,Y) - \dim \Ext(X,Y).$
\end{itemize}

We will restrict ourselves to the case where the genus $g \geq 1$; the case $g=0$ has been handled in Proposition \ref{proposition:AislesProjectiveLine}.  It then follows from the Riemann-Roch formula that $\dim_k \Ext(X,X) \geq 1$, thus every object has self-extensions.

Our main result is Theorem \ref{theorem:AislesCurves} below which states that $t$-structures are classified by torsion classes.

\begin{lemma}\label{lemma:NarrowOnCurves}
Let $\NN$ be a coreflective narrow subcategory of $\coh \bX$.  If $\NN$ contains a locally free sheaf, then $\NN$ contains all torsion sheaves.
\end{lemma}

\begin{proof}
Let $X \in \NN$ be a locally free sheaf.  Let $T$ be a simple (and hence torsion) sheaf (thus $T$ is supported on a single point $P \in \bX$), and suppose that $T \not\in \NN$.  It follows from the Riemann-Roch formula that $\Hom(X,T) \not= 0$ so that $T_\NN \not = 0$.

We know that $\Ext(-,T)|_{\NN} = 0$, so that $\Hom(-,T_\NN)$ is right exact and thus $T_\NN$ is an $\NN$-injective.  In particular, $\Ext(T_\NN,T_\NN) = 0$.  A contradiction since every object of $\NN$ has self-extensions.
\end{proof}

\begin{remark}
When $\bX = \bP^1$, then the last line of the proof of Lemma \ref{lemma:NarrowOnCurves} does not hold since $\Ext(\OO_{\bP^1}(n),\OO_{\bP^1}(n)) = 0$ for all $n \in \bZ$.  This will give rise to an additional set of $t$-structures on $\Db \coh \bP^1$, namely those of type II, as seen in Corollary \ref{corollary:AislesProjectiveLine}.
\end{remark}

\begin{proposition}\label{proposition:NarrowOnCurves}
Let $\NN$ be a coreflective narrow subcategory of $\coh \bX$.  Then $\NN$ is a torsion class in $\coh \bX$.  The wide closure $\WW$ of $\NN$ is $\coh \bX$ if and only if $\NN$ contains a nonzero locally free sheaf.
\end{proposition}

\begin{proof}
If $\NN$ does not contain a nonzero locally free sheaf (thus $\NN$ is a subcategory of the torsion sheaves), then the wide closure is contained in the torsion sheaves and $\wide \NN \not= \coh \bX$.  In this case, it is easily seen that $\NN$ is a torsion class in $\coh \bX$.

For the other case, assume that $\NN$ contains a locally free sheaf $X$.  By the previous lemma, $\NN$ contains all sheaves of finite length.  Let $T$ be a simple sheaf, supported on $P \in \bX$.

Since the wide closure $\WW$ of $\NN$ contains then the ample sequence $(X(nP))_{n \in \bZ}$ (in the sense of \cite{Polishchuk05}) we see that $\WW = \coh \bX$.  By Corollary \ref{corollary:NarrowIsTorsion} we see that $\NN$ is closed under quotient objects and extensions in $\coh \bX$.  Since $\NN$ is coreflective in $\coh \bX$, we know that $\NN$ is a torsion class (see Corollary \ref{corollary:NarrowIsTorsion}).
\end{proof}

\begin{corollary}
Let $\UU$ be an aisle, and let $\mu(\UU) = (\NN(k))_k$ be the associated (coreflective) narrow sequence.  If $\NN(k) \not= 0$, then $\NN(k+1) = \coh \bX$.
\end{corollary}

\begin{proof}
Without loss of generality, we may assume that $\NN(0) \not=0$.  By Lemma \ref{lemma:NarrowOnCurves}, we know that the torsion sheaves of $\coh \bX$ lie in $\NN(0)$.  To show that $\NN(1) = \coh \bX$, it suffices to show that $\NN(1)$ contains every locally free sheaf.  Therefore, seeking a contradiction, assume that $\NN(1) \not= \coh \bX$ and let $X \in \coh \bX$ be a nonzero locally free sheaf in $\NN(1)^{\perp_0}$ (thus $X$ lies in the torsionfree part of the torsion theory $(\NN(1),\NN(1)^{\perp_0})$).  For any nonzero torsion sheaf $T \in \NN(0)$, the Riemann-Roch formula implies that $\Ext(T,X) \not= 0$.

Let $Y = X[1] \in \Db \coh \bX$.  Note that since $\NN(0)$ contains all torsion sheaves, we have that $\Ext(\NN(0),X) \cong \Hom(\NN(0)[0],Y) \not= 0$ so that $Y_\UU \not= 0$.  Lemma \ref{lemma:HereditaryAdjoints} yields that $H_n(Y_\UU) = 0$ for all $n \not= 0,1$.  Since $X \in \NN(1)^{\perp_0}$ we know that $H_n(Y_\UU)=0$ for $n \not= 1$ as well.  We infer that $H_0(Y_\UU) \not= 0$.  This yields a nonzero right exact functor $\Hom(-,H_0(Y_\UU)): \NN(0) \to \mod k$ so that $H_0(Y_\UU)$ must be $\NN(0)$-injective.  But every object has self-extensions.  This shows that there cannot be an object $X$ as specified before and hence $\NN(1) = \coh \bX$.
\end{proof}

As a consequence, we obtain the following result (see originally \cite[Proposition 7.1]{GorodentsevKuleshovRudakov04}).

\begin{theorem}\label{theorem:AislesCurves}
The nondegenerate $t$-structures on $\Db \coh \bX$ up to suspension are induced by nonzero torsion classes in $\coh \bX$.
\end{theorem}

\begin{corollary}
The heart of every nondegenerate $t$-structure in $\Db (\coh \bX)$ is derived equivalent to $\coh \bX$.
\end{corollary}

A case where these torsion theories are explicitly known is the case of genus one curves (see \cite{GorodentsevKuleshovRudakov04}, and also \cite{Burban06, vanRoosmalen08}).  One then obtains an explicit classification of the $t$-structures of $\coh \bX$ where $\bX$ is an elliptic curve.

\providecommand{\bysame}{\leavevmode\hbox to3em{\hrulefill}\thinspace}
\providecommand{\MR}{\relax\ifhmode\unskip\space\fi MR }
% \MRhref is called by the amsart/book/proc definition of \MR.
\providecommand{\MRhref}[2]{%
  \href{http://www.ams.org/mathscinet-getitem?mr=#1}{#2}
}
\providecommand{\href}[2]{#2}

\end{document}